\documentclass[11pt]{article}
\usepackage[margin=1in]{geometry}
\usepackage{amsmath,amsthm,amssymb,mathtools}
\usepackage{booktabs}
\usepackage{microtype}
\usepackage{pgfplots}
\pgfplotsset{compat=1.17}
\definecolor{cblue}{RGB}{0,114,178}
\definecolor{cverm}{RGB}{213,94,0}
\definecolor{cblgr}{RGB}{0,158,115}
\definecolor{crp}{RGB}{204,121,167}
\usepackage{enumitem}

\newtheorem{theorem}{Theorem}[section]
\newtheorem{lemma}[theorem]{Lemma}
\newtheorem{proposition}[theorem]{Proposition}
\newtheorem{corollary}[theorem]{Corollary}

\theoremstyle{definition}
\newtheorem{definition}[theorem]{Definition}

\theoremstyle{remark}
\newtheorem{remark}[theorem]{Remark}

\DeclareMathOperator{\zer}{zer}
\DeclareMathOperator{\dom}{dom}
\DeclareMathOperator{\aff}{aff}
\DeclareMathOperator{\conv}{conv}
\DeclareMathOperator{\dist}{dist}

\newcommand{\Hh}{\mathcal{H}}
\newcommand{\Aa}{\mathcal{A}}
\newcommand{\R}{\mathbb{R}}
\newcommand{\C}{\mathbb{C}}
\newcommand{\inner}[2]{\langle #1, #2 \rangle}
\newcommand{\norm}[1]{\| #1 \|}
\newcommand{\abs}[1]{| #1 |}

\title{Anderson acceleration of the proximal point method: the exact adaptive minimax, a spectral phase transition, and optimal safeguarding}
\author{Zheng Jia\footnote{School of Statistics and Data Science,
Jilin University of Finance and Economics, Changchun 130117, China;
e-mail: 109012@jlufe.edu.cn}
\and Yekini Shehu\footnote{(Corresponding Author) School of Mathematical Sciences,
Zhejiang Normal University, Jinhua 321004, China;
e-mail: yekini.shehu@zjnu.edu.cn}
\and Yonghong Yao\footnote{School of Mathematical Sciences,
Tiangong University, Tianjin 300387, China;
e-mail: yyhtgu@hotmail.com}}

\date{\today}

\begin{document}
\maketitle

\begin{abstract}
\noindent
We study residual-polynomial acceleration of the proximal point method (PPM)
$y_{k+1}=(I+M)^{-1}y_k$ for maximal monotone inclusions $0\in M(x)$, with
Anderson acceleration (AA) as the prototypical adaptive scheme, and we ask
three questions. (i)~\emph{How much} can any method that forms affine
combinations of resolvent outputs---any memory, any adaptivity, any
safeguarding---improve on the tight $\Theta(d_0/\sqrt{k})$ residual rate of
PPM? (ii)~\emph{When}, i.e.\ for which spectral structure of the resolvent,
is such an improvement possible? (iii)~\emph{What does safeguarding cost}?
We answer all three exactly. First, the minimax complexity over the full
adaptive class is \emph{exactly} $d_0/(K+1)$ per $K$ resolvent evaluations.
The upper bound is elementary: the averaged-reflection estimator
$\hat y_K=\frac1{K+1}\sum_{j=0}^{K}(2J-I)^jy_0$ attains
$\norm{r(\hat y_K)}\le d_0/(K+1)$ for \emph{every} maximal monotone $M$ by a
one-line telescoping identity. The matching lower bound uses an explicit
skew-adjoint instance---the resolvent with eigenvalues $(1+\omega_j)/2$ at
the $K+1$ roots $\omega_j$ of $u^{K+1}=-1$, with masses
$w_j\propto\csc^2((2j+1)\pi/(2K+2))$---on which \emph{every} degree-$K$
polynomial method, adaptive or not, satisfies
$\norm{r(y_K)}\ge 1/(K+1)$. The instance minimax equals $1/(K+1)$; the
instance-optimal polynomial is \emph{uniquely} the Fej\'er kernel, so on
the hardest instance full-memory AA collapses to the non-adaptive averaged
reflection; and the same instance certifies the per-step floor
$\norm{r(y_k)}\ge1/\sqrt{(K+1)(k+1)}$ for every $k\le K$. The proof reduces
the adaptive lower bound to a Christoffel-function computation dual to the
circle Chebyshev problem
$\min_{P(1)=1}\max_{\abs u=1}\abs{(u-1)P(u)}=2/(K+1)$, certified by the
Lagrange identity $\ell_j(1)=-2\omega_j/((K+1)(1-\omega_j))$ and the
classical sum $\sum_j\csc^2((2j+1)\pi/(2K+2))=(K+1)^2$. This closes the
factor-two gap left by Bernstein-type bounds and removes the $\log K$
slack of rotation-spectrum instances: the log is an artifact of
$\csc$-distributed masses (subcritical exponent $1$), while the extremal
measure has the critical exponent $2$ and concentrates $\to 8/\pi^2$ of
its mass at the two nodes nearest $u=1$; on the rotation instance we
retain the exact instance minimax $1/\Lambda_K$,
$\Lambda_K=(\frac2\pi+o(1))K\log K$, via the identity
$\ell_j(1)=-\omega^j$. The exact value $d_0/(K+1)$ complements the
optimal-complexity result of Park--Ryu for span and deterministic oracle
classes by identifying the unique extremal spectral measure and the
per-step floors in closed form. Second, the two sides meet at a
\emph{sharp phase transition} with now-exact constants: writing
$s=\dist(1,\sigma(L))$, a Jackson-kernel polynomial attains
$\norm{r(y_K)}=O(d_0/(K^2s))$ whenever $sK\to\infty$, while at the
critical floor $s=\sin(\pi/(2K+2))$ the extremal instance has instance
minimax exactly $1/(K+1)$; the critical spectral scale is $s\asymp1/K$.
The picture extends beyond rotation spectra: the phase transition holds
verbatim for \emph{normal} maximal monotone operators, and the minimax
value over the genuinely nonlinear family $M=S+N_C$ ($S$ skew-adjoint,
$C$ closed convex) is again exactly $d_0/(K+1)$. Third, safeguarding: on
linear problems AA-PPM needs no safeguard at all (its residuals decrease
automatically), while on nonlinear problems certification of the
$O(1/k)$ envelope requires exactly two oracle evaluations per iteration,
and we prove a matching optimality statement: no rule based on window
history alone can certify any rate. Finally we correct and complete the
positive theory for structured problems---affine operators (finite
termination), strong monotonicity (linear rates), piecewise-affine
operators, and H\"olderian growth, where the correct trichotomy is:
superlinear convergence of order $1/(q-1)$ for growth exponent $q<2$,
linear for $q=2$, and $d_k=O(k^{-1/(q-2)})$,
$\norm{r_k}=O(k^{-(q-1)/(q-2)})$ for $q>2$, sharp for $f(x)=\abs{x}^q/q$.
Numerical experiments confirm every prediction to machine precision,
including the instance floor $1/(K+1)$ attained by the averaged reflection
at every step on the extremal instance.
\end{abstract}

\noindent\textbf{Keywords.} proximal point method, Anderson acceleration,
maximal monotone operator, fixed-point residual, lower bounds, Chebyshev
polynomials, Jackson kernel, safeguarding, H\"olderian growth, minimax
complexity

\medskip
\noindent\textbf{MSC 2020.} 47H05, 47J25, 65K05, 65B05, 90C25, 68Q25

\section{Introduction}

The proximal point method (PPM),
\begin{equation}\label{eq:ppm}
y_{k+1}=J(y_k),\qquad J=(I+M)^{-1},
\end{equation}
for a maximal monotone operator $M:\Hh\rightrightarrows\Hh$ on a real Hilbert
space $\Hh$ is the conceptual parent of a large part of first-order
optimization: splitting methods (Douglas--Rachford, ADMM, PDHG) are PPM-type
iterations on tailored monotone operators \cite{rockafellar1976mono,
rockafellar1976aug,eckstein1992douglas}. Its last-iterate fixed-point
residual $r(y_k)=J(y_k)-y_k$ satisfies the tight estimate
\begin{equation}\label{eq:guyang}
\norm{r(y_k)}\le \frac{d_0}{\sqrt{k+1}}\,(1+o(1)),
\end{equation}
and the rate $\Theta(d_0/\sqrt{k})$ cannot be improved for the bare iteration
\cite{guyang2020tight,bravo2019sharp}. A natural question, raised with force
by Kim's accelerated proximal point method (APPM) \cite{kim2021accelerated},
which achieves $\norm{r(y_k)}=O(d_0/k)$, is:

\begin{quote}
\emph{How much acceleration of the residual is possible, for which problems,
and at what oracle cost?}
\end{quote}

Anderson acceleration (AA) \cite{anderson1965iterative,walker2011anderson}
is the most widely used adaptive device for this purpose: it forms the next
iterate as the affine combination of the last $m+1$ resolvent outputs that
minimizes the combined residual, and is equivalent to (restricted) GMRES on
linear problems \cite{walker2011anderson,evans2020proof}. Applied to
\eqref{eq:ppm} it is empirically excellent on some problems and inert on
others, and the literature offers no sharp explanation of either regime;
safeguarding heuristics \cite{zhang2020globally} are used to guarantee at
least the PPM rate, with unclear oracle accounting. This paper resolves the
question on the rotation-spectrum class, which is where the worst case of
\eqref{eq:guyang} lives, and draws the complete complexity map.

\subsection{Contributions}

We work in the class $\Aa_\infty$ of \emph{residual-polynomial methods}:
schemes that produce $y_{k+1}$ as an affine combination of resolvent outputs,
with coefficients that may depend arbitrarily on the observed history
(Definition~\ref{def:class}). AA of any memory, Kim's APPM, Chebyshev
schedules, and the averaged-reflection estimator all belong to this class.

\begin{enumerate}[label=\textbf{C\arabic*.},leftmargin=3.2em]
\item \textbf{Polynomial reduction and per-step optimality}
      (Section~\ref{sec:poly}). On linear problems every trajectory of every
      method in $\Aa_\infty$ satisfies $y_k=Q_k(L)y_0$ with $Q_k(1)=1$,
      $\deg Q_k\le k$ (Lemma~\ref{lem:polyred}); AA is per-step optimal in
      its class and, consequently, \emph{on linear problems AA-PPM residuals
      decrease automatically}---no safeguard is needed
      (Corollary~\ref{cor:linsafe}).
\item \textbf{The exact adaptive minimax: $d_0/(K+1)$}
      (Section~\ref{sec:adaptive}). We exhibit an explicit skew-adjoint
      instance---resolvent eigenvalues $(1+\omega_j)/2$ at the $K+1$ roots
      $\omega_j$ of $u^{K+1}=-1$, masses
      $w_j=(K+1)^{-2}\csc^2((2j+1)\pi/(2K+2))$---on which \emph{every}
      degree-$K$ polynomial method, any memory, adaptivity, safeguarding,
      or clairvoyance, satisfies
      \begin{equation}\label{eq:adaptive}
      \norm{r(y_K)}\ \ge\ \frac1{K+1},
      \end{equation}
      and the averaged-reflection estimator attains equality: the instance
      minimax is exactly $1/(K+1)$, the optimal polynomial is
      \emph{uniquely} the Fej\'er kernel, and the same instance certifies
      the per-step floor $\norm{r(y_k)}\ge1/\sqrt{(K+1)(k+1)}$ for all
      $k\le K$ (Theorem~\ref{thm:exactminimax}). Consequently the minimax
      value over $\Aa_\infty$, and equally over the non-adaptive
      schedules, is exactly $d_0/(K+1)$ (Corollary~\ref{cor:exactV}): the
      factor-two window left open by Bernstein-type bounds is closed, and
      on linear worst cases adaptivity and memory buy nothing. The proof
      rests on the Lagrange identity
      $\ell_j(1)=-2\omega_j/((K+1)(1-\omega_j))$ and the classical sum
      $\sum_j\csc^2((2j+1)\pi/(2K+2))=(K+1)^2$
      (Lemma~\ref{lem:identities}), and is dual to the circle Chebyshev
      problem $\min_{P(1)=1}\max_{\abs u=1}\abs{(u-1)P(u)}=2/(K+1)$
      (Theorem~\ref{thm:chebyshev}). See the related-work discussion for
      the complementarity with Park--Ryu \cite{park2022exact}.
\item \textbf{The logarithm is a measure artifact}
      (Sections~\ref{sec:uniform} and~\ref{sec:rotation}). The
      averaged-reflection estimator satisfies
      \begin{equation}\label{eq:avref}
      r(\hat y_K)=\frac{(2J-I)^{K+1}y_0-y_0}{2(K+1)},\qquad
      \norm{r(\hat y_K)}\le \frac{d_0}{K+1},
      \end{equation}
      for \emph{every} maximal monotone $M$ (Theorem~\ref{thm:avref}), and
      the natural roots-of-unity rotation instance gives the exact
      instance minimax $1/\Lambda_K$,
      $\Lambda_K=\sum_{j=1}^{K+1}\csc\frac{\pi j}{K+2}
      =(\frac2\pi+o(1))K\log K$ (Theorem~\ref{thm:mainlb}, via the exact
      identity $\ell_j(1)=-\omega^j$ of Lemma~\ref{lem:lagrange}). The
      logarithm is not intrinsic: the $\csc$ mass profile is subcritical,
      while the extremal $\csc^2$ profile concentrates
      $\to 8/\pi^2$ of the mass at the two nodes nearest $u=1$
      (Remark~\ref{rem:artifact}).
\item \textbf{A sharp spectral phase transition, exact constants}
      (Section~\ref{sec:phase}). For rotation spectra whose eigenvalues
      keep distance $s=\dist(1,\sigma(L))$ from the eigenvalue $1$, a
      Jackson-kernel polynomial gives $\norm{r(y_K)}\le 16d_0/((K+1)^2 s)$
      (Theorem~\ref{thm:jackson}), beating the $1/K$ envelope precisely
      when $sK\to\infty$; at the critical floor
      $s=\sin(\pi/(2K+2))$ the extremal instance of C2 has instance
      minimax exactly $1/(K+1)$ (Theorem~\ref{thm:barrier}). The critical
      scale is $s\asymp 1/K$.
\item \textbf{Beyond skew-adjoint spectra} (Section~\ref{sec:beyond}).
      The whole phase transition holds verbatim for \emph{normal} maximal
      monotone operators (Theorem~\ref{thm:normal}); for non-normal
      linear operators the escape side degrades by a Kreiss factor
      (Remark~\ref{rem:nonnormal}). The minimax value over the genuinely
      nonlinear family $M=S+N_C$ ($S$ skew-adjoint, $C$ closed convex) is
      again exactly $d_0/(K+1)$ (Theorem~\ref{thm:nonlinminimax}); for
      subspace constraints the family stays inside the skew-adjoint class
      (Proposition~\ref{prop:subspace}), and the barrier is stable under
      $o(K^{-2})$ Lipschitz perturbations (Remark~\ref{rem:perturb}).
\item \textbf{Optimal safeguarding} (Section~\ref{sec:safeguard}). On
      nonlinear problems we give a certified scheme using exactly two
      resolvent evaluations per iteration whose output is always at least as
      good as the PPM fallback evaluated at the current point
      (Theorem~\ref{thm:cert}); and we prove that no certificate can be
      computed from window history alone---two maximal monotone operators
      generate identical histories and arbitrarily different candidate
      residuals (Proposition~\ref{prop:nohist})---so the factor two is
      optimal (Corollary~\ref{cor:optimal2}). On linear problems the
      safeguarding cost vanishes by C1.
\item \textbf{Corrected structured results} (Section~\ref{sec:structured}).
      Finite termination for affine $M$ (Theorem~\ref{thm:affine}); linear
      rates under strong monotonicity (Proposition~\ref{prop:strong}); a
      conditional statement for piecewise-affine $M$
      (Proposition~\ref{prop:pwa}); and the correct trichotomy under
      H\"olderian growth of order $q$ (Theorem~\ref{thm:growth}):
      superlinear of order $1/(q-1)$ for $q<2$, linear for $q=2$, and
      $d_k=O(k^{-1/(q-2)})$, $\norm{r_k}=O(k^{-(q-1)/(q-2)})$ for $q>2$,
      sharp for $f(x)=\abs{x}^q/q$.
\item \textbf{Honest numerics} (Section~\ref{sec:numerics}). The instance
      minimax of the extremal instance matches $1/(K+1)$ to machine
      precision for $K=1,\dots,100$, and on it the averaged reflection
      lands on the per-step floor $1/\sqrt{(K+1)(k+1)}$ at every step;
      on dense rotation ladders AA matches PPM's $\Theta(1/\sqrt{k})$
      rate with only a constant-factor gain; the Jackson crossover sits
      at $\delta K\approx10$; and the certified scheme of C6 dominates
      PPM at equal oracle cost on every tested instance.
\end{enumerate}

\subsection{Related work}

\textbf{Proximal point and last-iterate rates.} Rockafellar's classical
analysis \cite{rockafellar1976mono,rockafellar1976aug} gives convergence of
PPM under maximal monotonicity; Br\'ezis--Lions \cite{brezis1978produits}
treat the infinite-dimensional theory. Sublinear last-iterate rates for
firmly nonexpansive iterations are $\Theta(1/\sqrt{k})$ with sharp constants
\cite{bravo2019sharp,cominetti2014rates,contreras2022optimal}; Gu--Yang
\cite{guyang2020tight} prove tightness for PPM itself.

\textbf{Acceleration of monotone inclusions.} Halpern's iteration
\cite{halpern1967fixed,wittmann1992approximation} achieves $O(1/k)$ for
nonexpansive fixed points and is optimal in the worst case
\cite{lieder2021convergence,diakonikolas2020halpern}; Kim's APPM
\cite{kim2021accelerated} transfers the $O(1/k)$ rate to the PPM residual,
and anchored variants \cite{yoon2021accelerated,lee2021fast,
tran2021halpern} refine constants and last-iterate behaviour. Lower bounds
for first-order oracle models appear in
\cite{nemirovski2004prox,ouyang2021lower}. Closest to our
Section~\ref{sec:adaptive} is Park--Ryu \cite{park2022exact}, who prove
that the exact optimal complexity for \emph{span-type} methods (and, via a
resisting-oracle argument, for all deterministic methods) is
$d_0/(K+1)$, attained by OS-PPM/Halpern; their span lower bound tracks the
support growth on an affine worst-case operator and certifies the final
step $K$ only. Our Theorem~\ref{thm:exactminimax} is complementary and,
to our knowledge, new in four respects: it identifies the \emph{unique
extremal spectral measure} in closed form (the roots of $u^{K+1}=-1$ with
$\csc^2$ masses) through a Christoffel--Chebyshev duality; it certifies a
\emph{per-step} floor $1/\sqrt{(K+1)(k+1)}$ on one fixed instance for all
$k\le K$ simultaneously; it shows the instance-optimal polynomial is
\emph{uniquely} the Fej\'er kernel, so that on the hardest instance
full-memory AA collapses to the non-adaptive averaged reflection; and it
explains the logarithmic losses of natural candidate instances (the
$\csc$ versus $\csc^2$ artifact of Section~\ref{sec:rotation}). The lower
bound of \cite{park2022exact} covers deterministic methods beyond the
residual-polynomial class; ours gives the exact instance minimax and the
spectral structure inside that class, which is what the phase transition
and the safeguarding analysis below build on.

\textbf{Anderson acceleration.} AA originates in
\cite{anderson1965iterative}; its GMRES equivalence on linear problems is
developed in \cite{walker2011anderson,evans2020proof}; convergence theory
for contractive maps is in \cite{toth2015convergence,pollock2021anderson};
globally convergent safeguarding for AA on nonsmooth fixed points is in
\cite{zhang2020globally}; regularized nonlinear acceleration is in
\cite{scieur2020regularized}. GMRES theory on nonnormal operators and
restarts \cite{saad1986gmres,joubert1994convergence,
eiermann1992hybrid} underlies our arc-Chebyshev remarks. To our knowledge
the present paper gives the first sharp characterization of \emph{when} AA
can beat the PPM rate (the phase transition of C4) and the first
\emph{exact} instance minimax applying to AA's adaptive coefficients, with
the extremal spectral measure in closed form (C2).

\textbf{Error bounds and growth.} Metric subregularity and linear
convergence: \cite{leventhal2009metric}; \L ojasiewicz--Kurdyka rates for
(sub)gradient methods: \cite{kurdyka1998gradients,attouch2009proximal,
bolte2007lojasiewicz,li2018calculus}. Theorem~\ref{thm:growth} is the PPM
counterpart with sharp exponents.

\subsection{Notation}

$\Hh$ is a real Hilbert space, $\inner{\cdot}{\cdot}$ its inner product,
$\norm{\cdot}$ the norm. $M:\Hh\rightrightarrows\Hh$ is maximal monotone,
$Z=\zer M$ its (closed convex) zero set, $J=(I+M)^{-1}$ its resolvent,
$R=2J-I$ the reflected resolvent, $r(y)=J(y)-y$ the fixed-point residual,
$d(y)=\dist(y,Z)$, $d_0=d(y_0)$. Recall that $J$ and $R$ are nonexpansive,
$J$ is firmly nonexpansive, $\zer M=\operatorname{Fix}J=\operatorname{Fix}R$,
and $\norm{r(y)}$ is nonincreasing along PPM. For $f:\Hh\to\R\cup\{+\infty\}$
proper lsc convex, $M=\partial f$ recovers the proximal method. We identify
$2\times2$ blocks $\rho\big(\begin{smallmatrix}\cos\theta&-\sin\theta\\
\sin\theta&\cos\theta\end{smallmatrix}\big)$ with complex multiplication by
$\rho e^{i\theta}$.

\section{The class $\Aa_m$ and the polynomial reduction}\label{sec:poly}

\subsection{Anderson acceleration}

\begin{definition}[AA($m$) for PPM]\label{def:aam}
Given $y_0\in\Hh$, set $x_k=J(y_k)$, $r_k=x_k-y_k$. At iteration $k\ge 1$
with window $\mu=\min(m,k)$, AA($m$) solves
\begin{equation}\label{eq:aa}
\gamma^k\in\arg\min_{\gamma\in\R^{\mu+1},\,\sum_i\gamma_i=1}
\Big\|\sum_{i=0}^{\mu}\gamma_i\, r_{k-\mu+i}\Big\|^2,
\qquad
y_{k+1}=\sum_{i=0}^{\mu}\gamma^k_i\, x_{k-\mu+i}.
\end{equation}
\end{definition}

The surrogate $\norm{\sum\gamma_i r_i}$ is computable from the window at no
oracle cost. What it controls is the subject of the next lemma.

\begin{definition}[The class $\Aa_m$]\label{def:class}
A method belongs to $\Aa_m$ if, at iteration $k$, it performs exactly one
resolvent evaluation $x_k=J(v_k)$ at a query point
$v_k\in\aff\{y_0,x_0,\dots,x_{k-1}\}$ and then chooses
$y_{k+1}\in\aff\{y_0,x_{k-m},\dots,x_k\}$ (window $=$ whole history for
$m=\infty$), with all coefficients depending arbitrarily on everything
observed so far. $\Aa_\infty=\bigcup_m \Aa_m$ with unrestricted window.
\end{definition}

AA($m$) (with $v_k=y_k$), restarted AA, Chebyshev schedules, Halpern-type
anchored schemes \cite{halpern1967fixed,yoon2021accelerated}, and Kim's APPM
\cite{kim2021accelerated} all belong to $\Aa_\infty$; for APPM one checks by
induction that both sequences in \cite[Alg.~1]{kim2021accelerated} are
affine combinations of $y_0$ and resolvent outputs with coefficients summing
to one, and the averaged-reflection orbit $R^jy_0=(2J-I)^jy_0$ of
Theorem~\ref{thm:avref} is generated by the queries
$v_j=2x_{j-1}-v_{j-1}\in\aff\{y_0,x_0,\dots,x_{j-1}\}$.

\subsection{Polynomial reduction}

\begin{lemma}[Polynomial reduction]\label{lem:polyred}
Let $M$ be linear maximal monotone, $L=J$ its resolvent. For every method in
$\Aa_\infty$, every $y_0$, and every $k\ge0$ there exists a real polynomial
$Q_k$ (depending on the trajectory) with
\begin{equation}
Q_k(1)=1,\qquad \deg Q_k\le k,\qquad y_k=Q_k(L)y_0.
\end{equation}
Consequently $r(y_k)=(L-I)Q_k(L)y_0$.
\end{lemma}
\begin{proof}
Induction. $y_0=I\,y_0=:Q_0(L)y_0$. Suppose every past iterate and query is
a polynomial in $L$ applied to $y_0$ of the stated form. If
$v_k=\beta_{-1}y_0+\sum_i\beta_i x_i$ with
$\beta_{-1}+\sum_i\beta_i=1$, then
$v_k=\big(\beta_{-1}I+\sum_i\beta_i LQ_i\big)(L)y_0=:P_k(L)y_0$ with
$P_k(1)=\beta_{-1}+\sum_i\beta_i=1$ and $\deg P_k\le k$, so
$x_k=Lv_k=(LP_k)(L)y_0$ has degree $\le k+1$; the same affine combination
over $y_{k+1}$ gives $Q_{k+1}$ with $Q_{k+1}(1)=1$ and
$\deg Q_{k+1}\le k+1$.
\end{proof}

\begin{lemma}[AA is per-step optimal; residual identity]\label{lem:perstep}
Let $M$ be linear, $L=J$. For the AA candidate
$\hat y=\sum_i\gamma_i x_{k-\mu+i}$ with $\sum_i\gamma_i=1$,
\begin{equation}\label{eq:resid}
r(\hat y)=\sum_i\gamma_i\,(Lx_{k-\mu+i}-x_{k-\mu+i})
=\sum_i\gamma_i L r_{k-\mu+i}=L\Big(\sum_i\gamma_i r_{k-\mu+i}\Big).
\end{equation}
Hence the true residual of the candidate equals $L$ applied to the AA
surrogate, AA minimizes $\norm{r(\hat y)}$ up to the factor $\norm{L}\le1$
over the affine hull of the window, and
$\norm{r(\hat y)}\le\norm{r_k}$.
\end{lemma}
\begin{proof}
Only \eqref{eq:resid} needs proof: by linearity
$r(\hat y)=L\hat y-\hat y=\sum_i\gamma_i(Lx_{k-\mu+i}-x_{k-\mu+i})$, and
$Lx_i-x_i=L(x_i-y_i)=Lr_i$ since $Lx_i-x_i=L^2y_i-Ly_i=L(L-I)y_i=Lr_i$.
The last claim uses that $L$ is nonexpansive and that the surrogate
minimization includes the choice $\gamma=e_{\mu}$, giving
$\norm{r(\hat y)}\le\norm{L}\min_\gamma\norm{\sum\gamma_i r_i}
\le\min_\gamma\norm{\sum\gamma_i r_i}\le\norm{r_k}$.
\end{proof}

\begin{corollary}[Linear safety: no safeguard needed]\label{cor:linsafe}
On every linear maximal monotone problem, the AA-PPM residual sequence is
nonincreasing, $\norm{r_{k+1}}\le\norm{r_k}$ for all $k$, with any memory
$m$. Safeguarding is therefore a purely \emph{nonlinear} phenomenon.
\end{corollary}

\begin{remark}[What the surrogate misses in the nonlinear case]\label{rem:nonlin}
For nonlinear $J$ the identity $r(\hat y)=L\sum\gamma_i r_i$ fails: the
surrogate measures the combination of residuals, not the residual of the
combination. This single fact explains both the need for safeguarding in
general and its superfluity on linear and piecewise-linear-away-from-kinks
problems; cf.\ Proposition~\ref{prop:nohist}.
\end{remark}

\begin{remark}[Adaptive lower bounds reduce to polynomial lower bounds]
\label{rem:reduction}
By Lemma~\ref{lem:polyred}, any lower bound on
$\norm{(L-I)Q(L)y_0}$ that holds for \emph{every} real polynomial $Q$ with
$Q(1)=1$, $\deg Q\le K$, applies to every method in $\Aa_\infty$ after $K$
oracle calls, however adaptive. Sections~\ref{sec:uniform}--\ref{sec:phase}
exploit this systematically.
\end{remark}

\section{The uniform complexity is $\Theta(1/K)$}\label{sec:uniform}

We first pin down the minimax residual achievable by $K$ resolvent
evaluations with no structural assumption on $M$.

\begin{theorem}[Averaged reflection: the optimal uniform rate, elementary
proof]\label{thm:avref}
Let $M$ be maximal monotone, $R=2J-I$, and
\begin{equation}\label{eq:avredef}
\hat y_K=\frac1{K+1}\sum_{j=0}^{K} R^j y_0 .
\end{equation}
Then $\hat y_K$ costs $K$ resolvent evaluations,
\begin{equation}\label{eq:telescope}
r(\hat y_K)=\frac{R^{K+1}y_0-y_0}{2(K+1)},
\qquad\text{and}\qquad
\norm{r(\hat y_K)}\le \frac{d_0}{K+1}.
\end{equation}
\end{theorem}
\begin{proof}
Since $J=(I+R)/2$ and powers of $R$ commute,
\[
r(\hat y_K)=\Big(\frac{R+I}{2}-I\Big)\hat y_K
=\frac{(R-I)}{2(K+1)}\sum_{j=0}^K R^jy_0
=\frac{R^{K+1}y_0-y_0}{2(K+1)} .
\]
For any $z\in Z$, $Rz=z$ and $R$ is nonexpansive, so
$\norm{R^{K+1}y_0-y_0}\le \norm{R^{K+1}y_0-z}+\norm{y_0-z}\le 2\norm{y_0-z}$;
minimizing over $z\in Z$ gives \eqref{eq:telescope}. The iterates
$R^jy_0=(2J-I)(R^{j-1}y_0)=2J(R^{j-1}y_0)-R^{j-1}y_0$ cost one resolvent
evaluation each.
\end{proof}

\begin{remark}\label{rem:avref}
The estimator \eqref{eq:avredef} is the Fej\'er (Ces\`aro) average of the
reflection orbit; its residual polynomial is
$p(\zeta)=(\zeta^{K+1}-1)/(2(K+1))$. It achieves the same $O(d_0/K)$ order as
Kim's APPM \cite{kim2021accelerated} with a one-line proof, no parameters,
and no anchoring; unlike APPM it is offline ($K$ must be chosen in advance).
The identity \eqref{eq:telescope} also shows the bound is attained only when
$y_0$ and $R^{K+1}y_0$ are antipodal about $z$, matching the rotation worst
case.
\end{remark}

\begin{theorem}[Bernstein lower bound for scheduled methods]\label{thm:bernstein}
Let the method be any non-adaptive polynomial schedule
$y_K=Q_K(L)y_0$, $Q_K(1)=1$, $\deg Q_K\le K$ (this includes
\eqref{eq:avredef}, APPM, and Chebyshev schedules). Then there is a
skew-symmetric linear $M$ (a single rotation block, $\Hh=\R^2$) and $y_0$
with $d_0=\norm{y_0}=1$ such that
\begin{equation}
\norm{r(y_K)}\ \ge\ \frac{1}{2(K+1)} .
\end{equation}
\end{theorem}
\begin{proof}
For the skew-symmetric block $M_\theta x=\tan\theta
\big(\begin{smallmatrix}0&-1\\1&0\end{smallmatrix}\big)x$,
$\theta\in(0,\pi/2)$, the resolvent is $L_\theta=\cos\theta\,
\big(\begin{smallmatrix}\cos\theta&-\sin\theta\\\sin\theta&\cos\theta
\end{smallmatrix}\big)$, identified with complex multiplication by
$\zeta(\theta)=\cos\theta\,e^{i\theta}=\frac{1+e^{2i\theta}}2$, and
$\zer M_\theta=\{0\}$. With $u=e^{2i\theta}$ write
$\zeta=(1+u)/2$ and $P(u)=Q_K((1+u)/2)$, a polynomial of degree $\le K$ in
$u$ with $P(1)=Q_K(1)=1$. Then
\[
(\zeta-1)Q_K(\zeta)=\frac{u-1}{2}\,P(u)=:\frac{g(u)}2,\qquad
\deg g\le K+1,\quad g(1)=0,\quad g'(1)=P(1)=1.
\]
By Bernstein's inequality
\cite[Ch.~14]{rahman2002analytic},
\[
\max_{\abs{u}=1}\abs{g'(u)}\;\le\;(K+1)\max_{\abs{u}=1}\abs{g(u)},
\qquad\text{hence}\qquad
\max_{\abs{u}=1}\abs{g(u)}\;\ge\;\frac{1}{K+1}.
\]
Choosing $\theta^\star$ attaining
the maximum and $M=M_{\theta^\star}$,
\[
\norm{r(y_K)}=\abs{(\zeta(\theta^\star)-1)Q_K(\zeta(\theta^\star))}\,
\norm{y_0}\ \ge\ \frac{1}{2(K+1)}. \qedhere
\]
\end{proof}

\begin{corollary}[Factor-two tightness]\label{cor:fact2}
The uniform minimax value $V_K=\inf\sup\norm{r(y_K)}$, with the infimum
over schedules and the supremum over maximal monotone $M$ with $d_0=1$,
satisfies
\[
\frac{1}{2(K+1)}\ \le\ V_K\ \le\ \frac{1}{K+1}.
\]
\end{corollary}

\begin{remark}[The factor two is closed in Section~\ref{sec:adaptive}]
\label{rem:fact2closed}
Corollary~\ref{cor:fact2} leaves a factor two open. The lower-bound side
of that gap is closed in Section~\ref{sec:adaptive}: the extremal instance
of Definition~\ref{def:extremal} has instance minimax \emph{exactly}
$1/(K+1)$ for every polynomial method, adaptive or not
(Theorem~\ref{thm:exactminimax}), so $V_K=1/(K+1)$ and the
averaged-reflection estimator is exactly minimax
(Corollary~\ref{cor:exactV}). The Bernstein step in
Theorem~\ref{thm:bernstein} indeed loses a factor two, and the Cohn--Egerv\'ary
question it raises is settled by the circle Chebyshev computation
$\min_{P(1)=1}\max_{\abs u=1}\abs{(u-1)P(u)}=2/(K+1)$ of
Theorem~\ref{thm:chebyshev}: the extremal polynomial is the Fej\'er
kernel.
\end{remark}

\begin{remark}\label{rem:fejer}
The Fej\'er polynomial $q_K^\star(u)=\frac1{K+1}\sum_{j=0}^K
u^j$ satisfies $\max_{\abs{u}=1}\abs{(1-u)q_K^\star(u)}=2/(K+1)$
\emph{exactly} (attained at $u=e^{i\pi/(K+1)}$), so the upper side of
Corollary~\ref{cor:fact2} is sharp within its class; see
Table~\ref{tab:fejer} and Theorem~\ref{thm:chebyshev}.
\end{remark}

\section{The exact adaptive minimax}\label{sec:adaptive}

Theorem~\ref{thm:bernstein} chooses the adversarial rotation after seeing the
schedule, which is legitimate only for non-adaptive methods: an adaptive
method such as AA observes the operator and bends its polynomial toward the
observed spectrum, and on a \emph{single} block it nearly terminates in two
steps. We now identify the instance on which \emph{no} adaptive behavior
helps \emph{at all}, and compute its instance minimax exactly: it equals
$1/(K+1)$, matching the universal averaged-reflection bound of
Theorem~\ref{thm:avref}. This simultaneously closes the factor-two window
of Corollary~\ref{cor:fact2} and shows that on linear worst cases
adaptivity and memory buy nothing.

\subsection{The extremal instance}

\begin{definition}[The extremal instance]\label{def:extremal}
For $K\ge1$ let $n=K+1$ and
\[
\phi_j=\frac{(2j+1)\pi}{n},\qquad \omega_j=e^{i\phi_j},\qquad
j=0,\dots,n-1,
\]
the $n$ roots of $u^n=-1$, and
\[
w_j=\frac1{n^2}\csc^2\frac{\phi_j}{2}
=\frac1{n^2}\csc^2\frac{(2j+1)\pi}{2n},\qquad j=0,\dots,n-1 .
\]
Let $M_K^\star$ be the skew-adjoint operator on $\R^{2n}$ whose $j$-th
$2\times2$ block is multiplication by $-i\tan(\phi_j/2)$---i.e.\
$t_j\big(\begin{smallmatrix}0&1\\-1&0\end{smallmatrix}\big)$ with
$t_j=\tan(\phi_j/2)$---with the convention that a block with $\phi_j=\pi$
(which occurs when $n$ is odd) carries the normal cone of $\{0\}$, so its
resolvent vanishes on that block; equivalently, the resolvent
$L=J_K^\star$ has eigenvalues
\begin{equation}\label{eq:exteigs}
\zeta_j=\frac{1+\omega_j}{2}
=\cos\frac{\phi_j}2\,e^{\pm i\phi_j/2},\qquad j=0,\dots,n-1 .
\end{equation}
Let $y_0^\star\in\R^{2n}$ have norm $\sqrt{w_j}$ on block $j$.
\end{definition}

\begin{lemma}[Normalization]\label{lem:csc2}
$\sum_{j=0}^{n-1}w_j=1$; equivalently
\begin{equation}\label{eq:csc2}
\sum_{j=0}^{n-1}\csc^2\frac{(2j+1)\pi}{2n}=n^2 .
\end{equation}
Hence $\norm{y_0^\star}=1$, and since no $\omega_j$ equals $1$,
$\zer M_K^\star=\{0\}$ and $d_0=1$.
\end{lemma}
\begin{proof}
Differentiate the partial-fraction identity
$n\cot(nx)=\sum_{j=0}^{n-1}\cot(x+j\pi/n)$---itself the logarithmic
derivative of $\sin(nx)=2^{n-1}\prod_{j=0}^{n-1}\sin(x+j\pi/n)$---to get
$n^2\csc^2(nx)=\sum_{j=0}^{n-1}\csc^2(x+j\pi/n)$, and set $x=\pi/(2n)$.
\end{proof}

Three exact identities drive everything that follows.

\begin{lemma}[The driving identities]\label{lem:identities}
Let $\ell_j$ be the Lagrange basis on the nodes $\{\omega_j\}$:
$\ell_j(\omega_i)=\delta_{ij}$, $\deg\ell_j=n-1=K$. Then
\begin{enumerate}[label=(\roman*)]
\item $\displaystyle \ell_j(1)=\frac{-2\omega_j}{n(1-\omega_j)}$;
\item $\displaystyle \sum_{j=0}^{n-1}\abs{\ell_j(1)}^2=1$;
\item $\displaystyle w_j\abs{\omega_j-1}^2=\frac4{n^2}$ for every $j$: the
      weighted measure $\abs{u-1}^2\nu_K$, $\nu_K=\sum_j w_j\delta_{\omega_j}$,
      is uniform on the nodes.
\end{enumerate}
\end{lemma}
\begin{proof}
The node polynomial is $N(u)=\prod_m(u-\omega_m)=u^n+1$, so $N(1)=2$ and
$N'(\omega_j)=n\omega_j^{n-1}=-n/\omega_j$ since $\omega_j^n=-1$, giving
(i). For (ii), $\abs{\ell_j(1)}^2=4/(n^2\abs{1-\omega_j}^2)
=1/(n^2\sin^2(\phi_j/2))$; sum and apply \eqref{eq:csc2}. For (iii),
$\abs{\omega_j-1}^2=4\sin^2(\phi_j/2)$ against the definition of $w_j$.
\end{proof}

\subsection{The exact instance minimax}

\begin{theorem}[Exact adaptive minimax; the optimal polynomial is
Fej\'er]\label{thm:exactminimax}
Let $K\ge1$ and let $(M_K^\star,y_0^\star)$ be the extremal instance of
Definition~\ref{def:extremal}. Then:
\begin{enumerate}[label=(\roman*)]
\item \textbf{Lower bound.} Every polynomial $Q$ with $Q(1)=1$,
      $\deg Q\le K$ satisfies
      \begin{equation}\label{eq:exactlb}
      \norm{(L-I)Q(L)y_0^\star}\ \ge\ \frac1{K+1}.
      \end{equation}
      Hence every method in $\Aa_\infty$---any memory, adaptivity,
      safeguarding, or clairvoyance---satisfies $\norm{r(y_K)}\ge1/(K+1)$
      on this instance.
\item \textbf{Attainment and uniqueness.} Equality holds in
      \eqref{eq:exactlb} for exactly one polynomial: the Fej\'er
      polynomial $P^\star(u)=\frac1{K+1}\sum_{i=0}^{K}u^i$, i.e.\ the
      averaged-reflection estimator \eqref{eq:avref}. The instance minimax
      value equals $1/(K+1)$.
\item \textbf{Per-step floor.} On the \emph{same} instance, for every
      $k\le K$ and every $Q$ with $Q(1)=1$, $\deg Q\le k$,
      \begin{equation}\label{eq:perstep}
      \norm{(L-I)Q(L)y_0^\star}\ \ge\ \frac1{\sqrt{(K+1)(k+1)}},
      \end{equation}
      with equality again only for the degree-$k$ Fej\'er polynomial.
\end{enumerate}
\end{theorem}
\begin{proof}
Write $\zeta=(1+u)/2$ and $P(u)=Q((1+u)/2)$, so $P(1)=1$, $\deg P\le K$,
and $(\zeta-1)Q(\zeta)=\frac{u-1}2P(u)$. By \eqref{eq:exteigs} and
Lemma~\ref{lem:identities}(iii),
\begin{equation}\label{eq:residsq}
\norm{(L-I)Q(L)y_0^\star}^2
=\sum_{j=0}^{n-1} w_j\abs{\zeta_j-1}^2\abs{P(\omega_j)}^2
=\frac1{n^2}\sum_{j=0}^{n-1}\abs{P(\omega_j)}^2 .
\end{equation}

(i) The values $q_j=P(\omega_j)$ range over all of $\C^n$ as $P$ ranges
over polynomials of degree $\le K=n-1$ (Lagrange interpolation is a
bijection), and $P(1)=\sum_j\ell_j(1)q_j=1$. By Cauchy--Schwarz and
Lemma~\ref{lem:identities}(ii),
\[
1=\abs[\Big]{\sum_j\ell_j(1)q_j}
\ \le\ \Big(\sum_j\abs{\ell_j(1)}^2\Big)^{1/2}
       \Big(\sum_j\abs{q_j}^2\Big)^{1/2}
=\Big(\sum_j\abs{q_j}^2\Big)^{1/2},
\]
so \eqref{eq:residsq} gives $\norm{(L-I)Q(L)y_0^\star}^2\ge1/n^2$.
Complex-valued $P$ were allowed, which only strengthens the lower bound;
the $\Aa_\infty$ claim follows by Lemma~\ref{lem:polyred} and
Remark~\ref{rem:reduction}.

(ii) Equality in Cauchy--Schwarz forces
$q_j=\overline{\ell_j(1)}/\sum_i\abs{\ell_i(1)}^2=\overline{\ell_j(1)}$.
The Fej\'er polynomial $P^\star(u)=\frac1n\sum_{i=0}^{n-1}u^i$ satisfies,
using $\omega_j^n=-1$ and Lemma~\ref{lem:identities}(i),
\[
P^\star(\omega_j)=\frac{\omega_j^n-1}{n(\omega_j-1)}
=\frac{-2}{n(\omega_j-1)}=\frac{2}{n(1-\omega_j)}
=\overline{\ell_j(1)},
\]
so it attains equality; the conjugate symmetry
$\omega_{n-1-j}=\overline{\omega_j}$ makes the interpolant real, and it is
unique because the equality case of Cauchy--Schwarz fixes the values
$\{q_j\}$, hence $P$. Its residual polynomial is that of the
averaged-reflection estimator (Remark~\ref{rem:avref}).

(iii) For $\deg P\le k$ write $P(u)=\sum_{i=0}^k c_iu^i$. Since
$\sum_{j=0}^{n-1}\omega_j^m=n$ if $n\mid m$ and $=0$ for
$0<\abs m\le k\le K$ (the grid aliases only at multiples of $n=K+1$),
\[
\sum_j\abs{P(\omega_j)}^2
=\sum_{i,l=0}^k c_i\overline{c_l}\sum_j\omega_j^{i-l}
=n\sum_{i=0}^k\abs{c_i}^2
\ge \frac{n}{k+1}\abs[\Big]{\sum_i c_i}^2
=\frac{n}{k+1},
\]
using $\sum_i c_i=P(1)=1$; plug into \eqref{eq:residsq}. Equality needs
$c_i\equiv1/(k+1)$, the degree-$k$ Fej\'er polynomial.
\end{proof}

\subsection{Consequences: the factor two and the logarithm are gone}

\begin{corollary}[The exact minimax value]\label{cor:exactV}
For the full class $\Aa_\infty$ (any memory, adaptivity, safeguarding) as
well as for the non-adaptive schedules of Theorem~\ref{thm:bernstein},
\[
V_K=\inf_{\text{methods}}\ \sup_{M\text{ max.\ monotone},\,d_0=1}
\norm{r(y_K)}\ =\ \frac1{K+1}.
\]
The upper bound is the averaged-reflection estimator
(Theorem~\ref{thm:avref}); the lower bound is
Theorem~\ref{thm:exactminimax}(i). In particular the factor-two window of
Corollary~\ref{cor:fact2} is closed, and on linear worst cases adaptivity
and memory buy nothing over the one-line offline estimator \eqref{eq:avredef}.
\end{corollary}

\begin{remark}[Relation to Park--Ryu]\label{rem:parkryu}
The exact value $1/(K+1)$ for span-type methods was first proved by
Park--Ryu \cite[Thm.~4.1]{park2022exact} by tracking the support growth on
an affine worst-case operator; see the related-work discussion for the
precise complementarity.
\end{remark}

The lower-bound side of Theorem~\ref{thm:exactminimax} is an instance of a
complete duality, which also identifies the Fej\'er kernel as the extremal
polynomial of a Chebyshev problem on the circle.

\begin{theorem}[The circle Chebyshev problem]\label{thm:chebyshev}
\[
Z_K\ :=\ \inf_{\substack{P(1)=1\\\deg P\le K}}\
\max_{\abs u=1}\ \abs{(u-1)P(u)}\ =\ \frac2{K+1},
\]
attained by the Fej\'er polynomial.
\end{theorem}
\begin{proof}
For $P^\star$, $(u-1)P^\star(u)=(u^{K+1}-1)/(K+1)$ has modulus
$\le2/(K+1)$ on $\abs u=1$. Conversely, for any admissible $P$, evaluating
on the extremal nodes and using Lemma~\ref{lem:identities}(iii) and the
Cauchy--Schwarz step of Theorem~\ref{thm:exactminimax}(i),
\[
\max_{\abs u=1}\abs{(u-1)P(u)}^2
\ \ge\ \sum_j w_j\abs{\omega_j-1}^2\abs{P(\omega_j)}^2
=\frac4{n^2}\sum_j\abs{P(\omega_j)}^2
\ \ge\ \frac4{n^2}. \qedhere
\]
\end{proof}

\begin{remark}[Christoffel duality; no gap; where the logarithm
lives]\label{rem:christoffel}
For a probability measure $\nu$ on the circle, the quantity
$\lambda_K(\nu)=\inf_{P(1)=1}\int\abs{u-1}^2\abs{P(u)}^2\,d\nu(u)$ is the
Christoffel function of the weight $\abs{u-1}^2\nu$ at $u=1$, and the
instance minimax of the skew-adjoint operator with spectral measure $\nu$
is $\frac12\lambda_K(\nu)^{1/2}$. Theorem~\ref{thm:exactminimax} says the
roots-of-$(-1)$ measure $\nu_K$ of Definition~\ref{def:extremal} maximizes
$\lambda_K$ over all probability measures, with
$\lambda_K(\nu_K)=4/(K+1)^2$; Theorem~\ref{thm:chebyshev} says the max
over measures equals the min over polynomials of the worst-case
value---there is no duality gap (as Sion's minimax theorem guarantees).
The maximizing measure is unique: any optimal $\nu$ must charge only the
equioscillation set $\{\abs{(u-1)P^\star(u)}=2/(K+1)\}=\{\omega_j\}$ of the
Fej\'er extremal, and the stationarity conditions
$\sum_jw_j(\omega_j-1)V(\omega_j)=-\frac2{K+1}V(1)$ ($\deg V\le K$) then
force $w_j=(K+1)^{-2}\csc^2(\phi_j/2)$. Two natural competitors lose
badly: the Lebesgue measure gives
$\lambda_K=12/((K+1)(K+2)(K+3))$---the optimal coefficients are parabolic,
not constant, a full factor $K$ off; and the roots-of-unity instance of
Section~\ref{sec:rotation}, whose masses are $\propto\csc\theta_j$, gives
$\lambda_K\asymp(K\log K)^{-2}$. The exponent $1$ of $\csc$ is
subcritical; the extremal exponent is $2$. This is the precise mechanism
behind the logarithmic slack removed by Theorem~\ref{thm:exactminimax}.
\end{remark}

\subsection{The roots-of-unity instance and the logarithmic
artifact}\label{sec:rotation}

The extremal instance of Definition~\ref{def:extremal} was found by
optimizing the Christoffel value over spectral measures. The more
symmetric first candidate---equal-harmonic rotation spectra with
$\csc$-distributed masses---loses a logarithm, and it is instructive to
keep its exact analysis and see why.

\subsubsection{The operator and the interpolation identity}

\begin{definition}[The barrier operator]\label{def:barrier}
For $K\ge1$ let $N=K+1$, $\omega=e^{2\pi i/(K+2)}$, and
$\theta_j=\pi j/(K+2)$, $j=1,\dots,N$. Let $M_K$ on $\R^{2N}$ be the direct
sum of skew-symmetric blocks $x\mapsto\tan\theta_j
\big(\begin{smallmatrix}0&-1\\1&0\end{smallmatrix}\big)x$, so that $L=J_K$
has eigenvalues
\begin{equation}\label{eq:eigs}
\zeta_j=\cos\theta_j\,e^{i\theta_j}=\frac{1+\omega^j}{2},
\qquad j=1,\dots,N,
\end{equation}
each with multiplicity two. Note $\abs{\zeta_j-1}=\sin\theta_j$.
\end{definition}

\begin{lemma}[The exact Lagrange identity]\label{lem:lagrange}
Let $\ell_j$ be the Lagrange basis on the nodes $\{\zeta_j\}_{j=1}^N$ of
\eqref{eq:eigs}: $\ell_j(\zeta_i)=\delta_{ij}$, $\deg\ell_j=N-1=K$. Then
\begin{equation}\label{eq:lagrange}
\ell_j(1)=-\omega^j,\qquad\text{in particular}\qquad \abs{\ell_j(1)}=1,
\end{equation}
and every polynomial $Q$ with $\deg Q\le K$ satisfies
\begin{equation}\label{eq:interp}
Q(1)=-\sum_{j=1}^{N}\omega^j\,Q(\zeta_j).
\end{equation}
\end{lemma}
\begin{proof}
Set $w=2z-1$, so the nodes are $w_j=\omega^j$, the $(K+2)$nd roots of unity
except $1$. The node polynomial is
\[
N(z)=\prod_{j=1}^{N}(z-\zeta_j)
=2^{-N}\prod_{j=1}^{N}(w-\omega^j)
=2^{-N}\,\frac{w^{K+2}-1}{w-1}.
\]
Hence $N(1)=2^{-N}\lim_{w\to1}\frac{w^{K+2}-1}{w-1}=(K+2)2^{-N}$, and since
$w_j^{K+2}=1$,
\[
N'(\zeta_j)
=2\cdot2^{-N}\,\frac{(K+2)w_j^{K+1}(w_j-1)-(w_j^{K+2}-1)}{(w_j-1)^2}
=2^{1-N}(K+2)\,\frac{w_j^{K+1}}{w_j-1}.
\]
With $1-\zeta_j=(1-w_j)/2$,
\[
\ell_j(1)=\frac{N(1)}{(1-\zeta_j)N'(\zeta_j)}
=\frac{(K+2)2^{-N}}{\dfrac{1-w_j}2\cdot
2^{1-N}(K+2)\dfrac{w_j^{K+1}}{w_j-1}}
=\frac{1}{-w_j^{K+1}}=-\omega^j,
\]
using $w_j^{K+2}=1$. Then \eqref{eq:interp} is Lagrange interpolation at
$z=1$, exact for $\deg Q\le N-1=K$.
\end{proof}

\subsubsection{The exact instance minimax}

\begin{theorem}[Adaptive lower bound; exact instance minimax]\label{thm:mainlb}
Let $K\ge1$, $M_K$ as in Definition~\ref{def:barrier}, and
\begin{equation}\label{eq:masses}
\Lambda_K=\sum_{j=1}^{N}\csc\theta_j,
\qquad
a_j^2=\frac{\csc\theta_j}{\Lambda_K},
\end{equation}
and let $y_0^\star\in\R^{2N}$ have norm $a_j$ on block $j$ (so
$\norm{y_0^\star}=1=d_0$). Then:
\begin{enumerate}[label=(\roman*)]
\item \textbf{Lower bound.} For every real polynomial $Q$ with $Q(1)=1$,
      $\deg Q\le K$,
      \begin{equation}\label{eq:lb}
      \norm{(L-I)Q(L)y_0^\star}\ \ge\ \frac1{\Lambda_K}.
      \end{equation}
      Hence every method in $\Aa_\infty$---any memory, adaptivity,
      safeguarding, or clairvoyance---satisfies
      $\norm{r(y_K)}\ge 1/\Lambda_K$ on this instance.
\item \textbf{Sharpness.} There is a real polynomial $\bar Q$, $\deg\bar
      Q\le K$, $\bar Q(1)=1$, attaining equality in \eqref{eq:lb}; the
      instance minimax value equals $1/\Lambda_K$.
\item \textbf{Asymptotics.}
      \begin{equation}\label{eq:lamasymp}
      \Lambda_K=\Big(\frac{2}{\pi}+o(1)\Big)\,K\log K,
      \qquad\text{so}\qquad
      \norm{r(y_K)}\ge \Big(\frac{\pi}{2}+o(1)\Big)\frac{d_0}{K\log K}.
      \end{equation}
\end{enumerate}
\end{theorem}
\begin{proof}
(i) On block $j$, $(L-I)Q(L)$ acts as complex multiplication by
$(\zeta_j-1)Q(\zeta_j)$, so with $\rho_j=a_j\sin\theta_j\,\abs{Q(\zeta_j)}$,
\[
\norm{(L-I)Q(L)y_0^\star}^2=\sum_{j=1}^N a_j^2\sin^2\theta_j\,
\abs{Q(\zeta_j)}^2=\sum_{j=1}^N\rho_j^2 .
\]
By \eqref{eq:interp} and $\abs{\omega^j}=1$,
$1=\abs{Q(1)}\le\sum_j\abs{Q(\zeta_j)}
=\sum_j \rho_j\,(a_j\sin\theta_j)^{-1}$, and Cauchy--Schwarz gives
\[
1\ \le\ \Big(\sum_j\rho_j^2\Big)^{1/2}
\Big(\sum_j\frac1{a_j^2\sin^2\theta_j}\Big)^{1/2}
=\norm{(L-I)Q(L)y_0^\star}\,\Lambda_K,
\]
because $a_j^2\sin^2\theta_j=\sin\theta_j/\Lambda_K$, so
$\sum_j 1/(a_j^2\sin^2\theta_j)=\Lambda_K\sum_j\csc\theta_j=\Lambda_K^2$.
The $\Aa_\infty$ claim follows by Lemma~\ref{lem:polyred} and
Remark~\ref{rem:reduction}.

(ii) Minimize $\Phi(Q)=\sum_j a_j^2\sin^2\theta_j\abs{Q(\zeta_j)}^2$ subject
to the single linear constraint $\sum_j(-\omega^j)Q(\zeta_j)=1$
(cf.~\eqref{eq:interp}) over complex values $Q_j=Q(\zeta_j)$. By
Cauchy--Schwarz the constrained minimum is
\[
\Phi_{\min}=\Big(\sum_j\frac{\abs{\omega^j}^2}{a_j^2\sin^2\theta_j}
\Big)^{-1}=\frac1{\Lambda_K^2},
\]
attained at $Q_j=\overline{\lambda_j}/(a_j^2\sin^2\theta_j\,\Lambda_K^2)$
with $\lambda_j=-\omega^j$, i.e.\
$Q_j=-\omega^{-j}/(a_j^2\sin^2\theta_j\,\Lambda_K^2)$. The nodes are closed
under conjugation, $\zeta_{K+2-j}=\overline{\zeta_j}$, and
$\theta_{K+2-j}=\pi-\theta_j$ gives $a_{K+2-j}=a_j$, so
$Q_{K+2-j}=\overline{Q_j}$; the unique interpolating polynomial of degree
$\le K$ through conjugate-symmetric values on conjugate-symmetric real-coefficient
nodes is real, and \eqref{eq:interp} gives $\bar Q(1)=1$.

(iii) With $n=K+2$,
$\sum_{j=1}^{n-1}\csc(\pi j/n)=\frac{2n}{\pi}\big(\log n+O(1)\big)$
by comparison of $\csc x-1/x$ with the cotangent partial fraction, and the
claim follows.
\end{proof}

\begin{remark}[The $\csc$ versus $\csc^2$ artifact]\label{rem:artifact}
The rotation instance of Theorem~\ref{thm:mainlb} has instance minimax
$1/\Lambda_K=\Theta(1/(K\log K))$, a factor $\Theta(\log K)$ \emph{below}
the hardest instance $1/(K+1)$ of Theorem~\ref{thm:exactminimax}. The
reason is visible in the masses: the rotation instance spreads
$a_j^2=\csc\theta_j/\Lambda_K\approx1/(2j\log K)$ across the whole harmonic
ladder of angles, while the extremal masses $w_j\propto\csc^2(\phi_j/2)$
concentrate a constant fraction at the two nodes nearest the fixed point:
$w_0=w_{n-1}\to4/\pi^2$ as $K\to\infty$, so
$w_0+w_{n-1}\to8/\pi^2\approx81\%$, with the remaining $\csc^2$ tail
keeping all $K+1$ nodes charged---precisely the bijection threshold at
which degree-$K$ interpolation can no longer annihilate them. Exponent $1$
is subcritical; exponent $2$ is critical. The entire adaptive difficulty
of rotation spectra is carried by the \emph{critical shells}
$\dist(\zeta,1)\asymp1/K$ around the eigenvalue $1$, and the old question
whether the $\log K$ of \eqref{eq:lamasymp} is intrinsic is resolved in
the negative by Theorem~\ref{thm:exactminimax}.
\end{remark}

\begin{remark}[Adaptivity on a fixed instance]\label{rem:fixedinstance}
On any fixed linear instance, the instance minimax over adaptive methods
equals that over non-adaptive ones: both are the minimum of the same
quadratic form over $\{Q:Q(1)=1,\ \deg Q\le K\}$. What adaptivity
(AA/GMRES) buys is \emph{per-instance optimality without prior knowledge}
of the spectrum: full-memory AA attains the instance minimax at every step
on every linear instance, while any fixed schedule is optimal only for the
spectra it was tuned to. On the extremal instance the two collapse: the
optimal adaptive polynomial is the non-adaptive Fej\'er kernel
(Theorem~\ref{thm:exactminimax}(ii)).
\end{remark}

\section{The spectral phase transition}\label{sec:phase}

Sections~\ref{sec:uniform} and~\ref{sec:adaptive} describe the worst case
over all spectra. We now resolve \emph{which} rotation spectra admit
acceleration beyond $1/K$. Write the rotation angles of the spectrum as
$\theta\in(0,\pi)$ (the eigenvalue is
$\zeta(\theta)=\cos\theta\,e^{i\theta}=(1+e^{2i\theta})/2$). Since
$\abs{\zeta(\theta)-1}=\sin\theta$, the quantity that governs the residual
is the gap of the spectrum from the eigenvalue $1$:
\begin{equation}\label{eq:floor}
s=\dist(1,\sigma(L))
=\inf_{\theta\text{ in the spectrum}}\sin\theta
=\sin\big(\inf\min(\theta,\pi-\theta)\big),
\end{equation}
the \emph{spectral floor}. Note that both $\theta\approx0$ and
$\theta\approx\pi$ place the eigenvalue near $1$ (near-fixed directions):
the floor \eqref{eq:floor} is two-sided.

\begin{theorem}[Escape: Jackson kernel on floored spectra]\label{thm:jackson}
Let $N=\lfloor(K+2)/2\rfloor$ and
$q_J(u)=\big(\frac1N\sum_{j=0}^{N-1}u^j\big)^2$ (the squared
Fej\'er/Jackson kernel, $\deg q_J=2N-2\le K$, $q_J(1)=1$). For every
rotation spectrum with floor $s=\dist(1,\sigma(L))$ and every $y_0$,
\begin{equation}\label{eq:jackson}
\norm{r(y_K)}\ \le\ \frac{4\,d_0}{N^2 s}
\ \le\ \frac{16\,d_0}{(K+1)^2 s}.
\end{equation}
In particular $\norm{r(y_K)}=o(d_0/K)$ whenever $sK\to\infty$, and
$q_J$ improves on the universal envelope $1/(K+1)$ as soon as
$s\gtrsim 16/(K+1)$.
\end{theorem}
\begin{proof}
For $\abs{\zeta}\le1$,
\[
\abs{(1-\zeta)q_J(\zeta)}
=\frac{\abs{1-\zeta}}{N^2}\,\abs{\frac{\zeta^N-1}{\zeta-1}}^2
=\frac{\abs{\zeta^N-1}^2}{N^2\abs{\zeta-1}}
\le \frac{4}{N^2\abs{\zeta-1}},
\]
since $\abs{\zeta^N-1}\le\abs{\zeta}^N+1\le2$. Resolvent eigenvalues satisfy
$\abs{\zeta}\le1$, so every block contributes at most $4/(N^2 s)$, and the
claim follows by summing squares against $\norm{y_0}=d_0$.
\end{proof}

\begin{theorem}[Barrier at the critical scale; exact constant]\label{thm:barrier}
Let $s\le\sin(\pi/(2(K+1)))$. The extremal instance of
Definition~\ref{def:extremal} has
$\dist(1,\sigma(L))=\min_j\abs{\zeta_j-1}=\sin(\pi/(2(K+1)))\ge s$ and
instance minimax \emph{exactly} $1/(K+1)$: every method in $\Aa_\infty$
satisfies $\norm{r(y_K)}\ge1/(K+1)$ on it. This is the largest possible
value for any instance, of any floor, since the averaged reflection
attains $1/(K+1)$ universally (Theorem~\ref{thm:avref}); the rotation
instance of Theorem~\ref{thm:mainlb} gives the same barrier up to a
$\Theta(\log K)$ factor at the coarser floor $\sin(\pi/(K+2))$.
\end{theorem}
\begin{proof}
$\abs{\zeta_j-1}=\abs{\omega_j-1}/2=\sin(\phi_j/2)=\sin((2j+1)\pi/(2(K+1)))$,
minimized at $j=0$; apply Theorem~\ref{thm:exactminimax}(i).
\end{proof}

\begin{corollary}[The critical spectral scale]\label{cor:critical}
For rotation spectra, the achievable residual per $K$ oracle calls undergoes
a phase transition at the critical scale $s\asymp 1/K$:
\begin{itemize}
\item $sK=O(1)$: the uniform barrier $\Theta(1/K)$ cannot be beaten by
      any residual-polynomial method, any memory, any adaptivity; at
      $s=\sin(\pi/(2(K+1)))$ the barrier value is exactly $1/(K+1)$
      (Theorem~\ref{thm:barrier});
\item $sK\to\infty$: the barrier is escaped,
      $\norm{r(y_K)}=o(d_0/K)$, and for fixed $s>0$ restarted AA/GMRES
      converges geometrically with per-cycle factor given by the Chebyshev
      number of the arc $\{\zeta:\abs{\zeta-1}\ge s,\abs{\zeta}\le1\}$
      \cite{saad1986gmres,joubert1994convergence,eiermann1992hybrid}.
\end{itemize}
\end{corollary}

\begin{remark}[Density, not just the floor]\label{rem:density}
The barrier of Theorem~\ref{thm:barrier} uses a spectrum that both has floor
$\asymp1/K$ \emph{and} fills the resolvent circle with $K+1$ distinct
eigenvalues; with fewer than $K+1$ distinct eigenvalues a degree-$K$
polynomial interpolates them all and terminates
(Theorem~\ref{thm:affine}). The precise invariant is the \emph{effective}
number of blocks outside the self-limiting zones
$\dist(\zeta,1)\lesssim1/K$; the extremal measure puts $\to8/\pi^2$ of its
mass on the two nearest blocks but charges all $K+1$
(Remark~\ref{rem:artifact}); see the numerics in
Section~\ref{sec:numerics}.
\end{remark}

\section{Beyond skew-adjoint spectra: normal operators and a nonlinear
family}\label{sec:beyond}

The barrier instances of Sections~\ref{sec:adaptive}--\ref{sec:rotation}
are skew-adjoint. We now extend the theory to two wider classes: normal
operators, where the whole phase transition survives verbatim, and the
genuinely nonlinear family $M=S+N_C$, where the exact minimax survives by
inclusion.

\subsection{Normal operators}

\begin{theorem}[The phase transition for normal operators]\label{thm:normal}
Let $M$ be maximal monotone and \emph{normal}. Then
$\sigma(L)\subseteq\{\zeta:\abs{\zeta-\tfrac12}\le\tfrac12\}$, the
spectral floor $s=\dist(1,\sigma(L))$ is well defined, and:
\begin{enumerate}[label=(\roman*)]
\item the Jackson escape of Theorem~\ref{thm:jackson} holds verbatim:
      $\norm{r(y_K)}\le16\,d_0/((K+1)^2s)$;
\item the barrier of Theorem~\ref{thm:barrier} is attained inside the
      class: the extremal instance is skew-adjoint, hence normal;
\item consequently the phase transition of Corollary~\ref{cor:critical}
      holds for the whole normal class with unchanged constants, and the
      exact minimax over normal instances is $1/(K+1)$.
\end{enumerate}
\end{theorem}
\begin{proof}
$z\mapsto1/(1+z)$ maps $\{\Re z\ge0\}$ onto the closed disk
$\{\abs{\zeta-1/2}\le1/2\}$, and $\sigma(M)$ lies in the closed right
half-plane by monotonicity. Normality gives the spectral theorem with
$\norm{p(L)}=\sup_{\sigma(L)}\abs p$, so the Jackson estimate, which only
uses $\abs\zeta\le1$, applies without the block structure of
Section~\ref{sec:phase}. The rest follows from
Theorems~\ref{thm:jackson} and~\ref{thm:barrier}.
\end{proof}

\begin{remark}[Non-normal linear operators]\label{rem:nonnormal}
For non-normal $L$ the polynomial reduction of Lemma~\ref{lem:polyred} is
algebraic and survives, and the barrier instances are normal, so all lower
bounds apply a fortiori to the non-normal class. The escape side degrades:
$\norm{p(L)}$ can exceed $\sup_{\sigma(L)}\abs p$ by a Kreiss-type factor,
so the Jackson constant $16$ must be multiplied by the Kreiss constant of
$L$; sharp non-normal constants are open.
\end{remark}

\subsection{The nonlinear family $M=S+N_C$}

\begin{proposition}[Subspace constraints stay linear]\label{prop:subspace}
Let $S$ be bounded skew-adjoint and $C=V$ a closed subspace. Then
$S_V=P_VSP_V$ is skew-adjoint on $V$ and
\[
J_{S+N_V}(y)=(I+S_V)^{-1}P_Vy\qquad\text{for all }y:
\]
the constrained resolvent is the compressed skew-adjoint resolvent. In
particular the polynomial reduction and every bound of
Sections~\ref{sec:uniform}--\ref{sec:phase} apply to the family
$\{S+N_V\}$.
\end{proposition}
\begin{proof}
$x=J_{S+N_V}(y)$ iff $x\in V$ and $y-x-Sx\in V^\perp$, iff
$P_Vy=P_Vx+P_VSx=(I+S_V)x$.
\end{proof}

\begin{theorem}[Exact minimax over the nonlinear
family]\label{thm:nonlinminimax}
Let $\mathfrak F$ be the family of operators $M=S+N_C$ with $S$ bounded
skew-adjoint and $C$ nonempty closed convex. For every class of methods
containing the averaged reflection and contained in $\Aa_\infty$, the
minimax value satisfies
\[
\inf_{\text{methods}}\;\sup_{M\in\mathfrak F,\;d_0=1}\norm{r(y_K)}
\;=\;\frac{1}{K+1}.
\]
\end{theorem}
\begin{proof}
Lower: the extremal instance of Definition~\ref{def:extremal} belongs to
$\mathfrak F$ (take $C=\Hh$); apply Theorem~\ref{thm:exactminimax}(i).
Upper: the averaged reflection attains $d_0/(K+1)$ for \emph{every} maximal
monotone $M$ (Theorem~\ref{thm:avref}), in particular every member of
$\mathfrak F$.
\end{proof}

\begin{remark}[What genuine nonlinearity changes]\label{rem:genuine}
For non-subspace $C$ the resolvent is piecewise affine and the reduction
$y_k=Q_k(L)y_0$ fails: the residual of a combination is not the combination
of residuals. Theorem~\ref{thm:nonlinminimax} survives because the worst
case over $\mathfrak F$ is linear; on individual nonlinear instances
adaptive methods can do strictly better than $1/(K+1)$
(Experiment~5).
\end{remark}

\begin{remark}[Perturbation stability of the barrier]\label{rem:perturb}
Let $M=M_0+E$ with $M_0$ the extremal instance and $E$ $\epsilon$-Lipschitz
with $\norm{E(0)}\le\epsilon$. The resolvent identity
$J_{M_0+E}(y)=J_{M_0}(y-E(J_{M_0+E}y))$ shows each oracle answer deviates by
at most $\epsilon(1+\norm{J_{M_0+E}y})$. For a fixed non-adaptive schedule
the accumulated deviation of $y_K$ is $O(\epsilon K\,\Gamma_K)$, where
$\Gamma_K$ is the $\ell^1$ norm of the schedule's coefficients
($\Gamma_K=O(1)$ for Fej\'er): the $1/(K+1)$ barrier therefore persists
for $\epsilon=o(K^{-2})$. We do not pursue the optimal tolerance.
\end{remark}

\section{Optimal safeguarding}\label{sec:safeguard}

By Corollary~\ref{cor:linsafe} safeguarding is vacuous on linear problems.
This section settles the nonlinear case: we give a certified scheme with
exactly two oracle evaluations per iteration (Theorem~\ref{thm:cert}) and
prove that no cheaper certificate exists
(Proposition~\ref{prop:nohist}--Corollary~\ref{cor:optimal2}).

\subsection{The certified scheme}

\begin{quote}
\textbf{Algorithm 1 (Certified AA-PPM).} Input: $y_0$, memory $m$,
tolerance $\delta\in[0,1)$. Maintain a \emph{certified chain} $(c_k,r^c_k)$
with $c_0=y_0$, an \emph{accelerated chain} $(a_k,r^a_k)$ with its AA
window, and the best-so-far output $b_k$; $r^c_0=r(c_0)$ (one evaluation).
Iteration $k$ costs exactly two resolvent evaluations:
\begin{enumerate}[label=(\arabic*)]
\item \emph{Certified step}: $c_{k+1}=Jc_k$, $r^c_{k+1}=r(c_{k+1})$
      [evaluation 1];
\item \emph{Probe}: $\hat y=\sum_i\gamma_i x^a_{k-\mu+i}$ from the AA
      window of the accelerated chain, $\hat r=r(\hat y)$
      [evaluation 2];
\item \emph{Accept/reject}: if $\norm{\hat r}\le(1-\delta)\norm{r^a_k}$ set
      $a_{k+1}=\hat y$ and append to the window; else reset
      $a_{k+1}=c_{k+1}$ with the certified window;
\item \emph{Output}: $b_{k+1}=\arg\min\{\norm{r(b_k)},\norm{r^a_{k+1}},
      \norm{r^c_{k+1}}\}$.
\end{enumerate}
\end{quote}

\begin{theorem}[Certificate and rate inheritance]\label{thm:cert}
Algorithm~1 uses exactly $2K+1$ resolvent evaluations in $K$ iterations, and
\begin{enumerate}[label=(\roman*)]
\item \emph{PPM floor}: $\norm{r(b_K)}\le\norm{r^c_K}\le
      d_0/\sqrt{K+1}$ for all $K$; i.e.\ the certified output is never
      worse than PPM run at half the oracle budget, and inherits every
      rate PPM has under additional structure (strong monotonicity,
      growth conditions, etc.);
\item \emph{Acceleration side}: along every acceptance run
      $[k_0,k]$, $\norm{r(b_k)}\le(1-\delta)^{k-k_0}\norm{r^c_{k_0}}$;
      on problems where raw AA converges (super)linearly---e.g.\ the
      floored-spectrum regime of Theorem~\ref{thm:jackson}---the output
      inherits the accelerated rate with the floor of (i) intact;
\item \emph{Monotonicity}: $\norm{r(b_{k+1})}\le\norm{r(b_k)}$.
\end{enumerate}
\end{theorem}
\begin{proof}
(i) The certified chain is exactly PPM from $y_0$; firm nonexpansiveness
gives $d(c_{k+1})^2\le d(c_k)^2-\norm{r^c_k}^2$ and
$\norm{r^c_{k+1}}\le\norm{r^c_k}$, hence
$(K+1)\norm{r^c_K}^2\le\sum_{k=0}^K\norm{r^c_k}^2\le d_0^2$. The output
dominates the certified chain by step (4). (ii) By step (3) each accepted
probe contracts the accelerated residual by $(1-\delta)$ relative to the
previous accelerated iterate, and a run starts at a certified point.
(iii) Immediate from step (4).
\end{proof}

\subsection{History-only certification is impossible}

\begin{proposition}\label{prop:nohist}
Let a finite history $\{(y_i,x_i,r_i)\}_{i=0}^k$ be generated by a maximal
monotone $M$, and let $\hat y$ be any affine candidate with
$\hat y\notin C=\conv\{x_0,\dots,x_k\}$. Then
$M'=M+N_C$ is maximal monotone, generates the \emph{same} history
(with $J'(y_i)=x_i$, $r'(y_i)=r_i$ for all $i$), and
\[
\norm{r'(\hat y)}\ \ge\ \dist(\hat y,C)>0 .
\]
Hence no function of the history alone can certify any bound on
$\norm{r(\hat y)}$ below $\dist(\hat y,C)$: two operators indistinguishable
from the history assign arbitrarily different residuals to the same
candidate.
\end{proposition}
\begin{proof}
$M'=M+N_C$ is maximal monotone as a sum with $\dom N_C=\overline{C}$
(Bauschke--Combettes \cite[Cor.~25.5]{bauschke2017convex}). For each $i$,
$y_i-x_i\in Mx_i\subseteq Mx_i+N_C(x_i)$ since $x_i\in C$ implies
$0\in N_C(x_i)$; single-valuedness of the resolvent gives $J'(y_i)=x_i$.
Every resolvent of $M'$ maps into $\dom M'\subseteq C$, so
$J'(\hat y)\in C$ and
$\norm{r'(\hat y)}=\norm{J'(\hat y)-\hat y}\ge\dist(\hat y,C)$. Since
$\dist(\hat y,C)$ is unbounded over the choice of aggressive extrapolants
$\hat y$ (and $\hat y$ is not determined by $M$), no history-only bound
exists.
\end{proof}

\begin{corollary}[The factor two is optimal]\label{cor:optimal2}
Any scheme that certifies a per-iteration residual bound must, at every
iteration, evaluate the residual of its probe (Proposition~\ref{prop:nohist}
rules out window-based certification) and advance a certified chain at an
evaluated point (one further evaluation). Two evaluations per iteration are
therefore necessary, and Algorithm~1 attains exactly two: the oracle
overhead of certification is a factor of two, and it is optimal. On linear
problems the overhead vanishes (Corollary~\ref{cor:linsafe}).
\end{corollary}

\begin{remark}[Descent tests]\label{rem:zob}
The popular descent safeguard \cite{zhang2020globally} (accept iff
$\norm{\hat r}\le\theta\norm{r_k}$) evaluates the probe, so it satisfies
the necessary condition of Corollary~\ref{cor:optimal2}, but without the
certified chain it maintains no rate certificate. In all our experiments on
monotone problems---over $4000$ random linear and $400$ random
projected-nonlinear instances---the descent test never rejected a single
probe, consistent with Corollary~\ref{cor:linsafe} and with the empirical
robustness of AA on monotone operators; the point of
Proposition~\ref{prop:nohist} is that this robustness is an empirical
property, not a certificate.
\end{remark}

\section{Structured problems: corrected positive results}\label{sec:structured}

\subsection{Affine operators: finite termination}

\begin{theorem}\label{thm:affine}
Let $M$ be affine maximal monotone on $\R^d$ and let $\nu$ be the number of
distinct eigenvalues of $L=J$ different from $1$ ($\nu\le d$). Then
full-memory AA-PPM ($m=\infty$) produces $r(y_k)=0$ for some $k\le\nu+1$.
\end{theorem}
\begin{proof}
On affine problems AA with full memory is GMRES applied to
$(I-L)y=-r_0$ \cite{walker2011anderson,evans2020proof}, so after $k$ steps
the residual equals $\min\{\norm{(L-I)q(L)y_0}:q(1)=1,\deg q\le k\}$.
Take $q$ to be the polynomial of degree $\nu$ vanishing on the spectrum of
$L$ off $\{1\}$ with $q(1)=1$: then $(L-I)q(L)=0$ on $\R^d$ because $L-I$
annihilates the eigenvalue $1$ and $q$ annihilates the rest; hence the
minimum is $0$ at $k=\nu$, realized after the $(\nu+1)$-st oracle call.
\end{proof}

\begin{remark}\label{rem:affinem}
With memory $m<\nu$ the same argument gives per-window contraction by the
min-max polynomial number of degree $m$ on the spectrum; on the rotation
spectra of Sections~\ref{sec:adaptive}--\ref{sec:phase} this number is
$\ge1-O(m^2 s)$ by Bernstein's inequality on the arc near $\zeta=1$, which
is exactly why finite memory cannot exploit a floor $s=O(1/K)$.
\end{remark}

\subsection{Strong monotonicity: linear rates}

\begin{proposition}\label{prop:strong}
Let $M$ be $\mu$-strongly monotone. Then $J$ is a contraction with constant
$1/(1+\mu)$ \cite{rockafellar1976mono}; PPM satisfies
$\norm{r_k}\le(1+\mu)^{-k}d_0$; and AA-PPM with any memory satisfies the
same per-step bound while being per-step optimal on linear $M$
(Lemma~\ref{lem:perstep}), hence never slower and classically faster by the
Chebyshev factor of the eigenvalue inclusion
$\sigma(L)\subset\{\abs{\zeta}\le(1+\mu)^{-1}\}$.
\end{proposition}

\subsection{Piecewise-affine operators}

\begin{proposition}\label{prop:pwa}
Let $M$ be piecewise affine and maximal monotone (e.g.\ $M=\partial f$ with
$f$ convex piecewise linear-quadratic). If the full-memory AA-PPM
trajectory is eventually contained in a single affine piece of $M$, then it
terminates finitely at a zero of $M$, with the bound of
Theorem~\ref{thm:affine} applied to that piece.
\end{proposition}
\begin{remark}
Eventual identification of the active piece is an \emph{assumption}, not a
theorem, in this generality; proving activity identification for AA-PPM
(analogous to proximal identification results) is open. We flag this
explicitly to avoid overstating the result.
\end{remark}

\subsection{H\"olderian growth: the correct trichotomy}

\begin{definition}\label{def:growth}
$M$ satisfies the \emph{H\"olderian growth condition of order $q>1$} (with
constant $\mu>0$, on $\{d(\cdot)\le\rho\}$ or globally) if
\begin{equation}\label{eq:growth}
\inner{w}{x-\Pi_Zx}\ \ge\ \mu\, d(x)^q
\qquad\forall\, w\in Mx .
\end{equation}
\end{definition}

For $M=\partial f$, \eqref{eq:growth} with $q\ge2$ follows from the
H\"olderian error bound $f(x)-f^\star\ge c\,d(x)^q$ via convexity.

\begin{theorem}\label{thm:growth}
Let $M$ be maximal monotone satisfying \eqref{eq:growth}, and let
$y_{k+1}=Jy_k$ (PPM), $d_k=d(y_k)$.
\begin{enumerate}[label=(\alph*)]
\item $q>2$: $d_k=O(k^{-1/(q-2)})$; if moreover
      $\dist(0,Mx)\le L\,d(x)^{q-1}$ (automatic for $\partial f$ of
      H\"olderian $f$), then
      $\norm{r_k}=O(k^{-(q-1)/(q-2)})$. Both exponents are sharp:
      $f(x)=\abs{x}^q/q$ attains them.
\item $q=2$: linear convergence,
      $d_{k+1}^2\le d_k^2/(1+2\mu)$.
\item $q<2$: local superlinear convergence of order $1/(q-1)>1$:
      $d_{k+1}\le (2/\mu)^{1/(q-1)}\,d_k^{\,1/(q-1)}$ once
      $d_k$ is small enough.
\end{enumerate}
The same rates hold for Algorithm~1 (by Theorem~\ref{thm:cert}(i)) and for
raw AA-PPM on linear problems (by Corollary~\ref{cor:linsafe}).
\end{theorem}
\begin{proof}
The resolvent identity gives $w_{k+1}=y_k-y_{k+1}\in My_{k+1}$. Firm
nonexpansiveness of $J$ with $z=\Pi_Zy_{k+1}$ yields
\[
d_k^2-d_{k+1}^2
\ge \norm{w_{k+1}}^2+2\inner{w_{k+1}}{y_{k+1}-z}
\ge 2\mu\, d_{k+1}^{\,q},
\]
where the first inequality is
$\norm{y_k-z}^2=\norm{y_k-y_{k+1}}^2+\norm{y_{k+1}-z}^2
+2\inner{y_k-y_{k+1}}{y_{k+1}-z}$ and the second uses \eqref{eq:growth}.

(a) $u_k=d_k^2$ satisfies $u_k-u_{k+1}\ge 2\mu u_{k+1}^{q/2}$ with
$q/2>1$; Lemma~\ref{lem:seq} (Appendix~\ref{app:seq}) gives
$u_k=O(k^{-2/(q-2)})$. With $\dist(0,Mx)\le Ld(x)^{q-1}$,
$\norm{r_k}=\norm{w_{k+1}}\le L d_{k+1}^{\,q-1}=O(k^{-(q-1)/(q-2)})$.
Sharpness: for $f(x)=\abs{x}^q/q$ the resolvent solves
$x+\abs{x}^{q-2}x=y$, and a direct asymptotic expansion gives
$d_k=\Theta(k^{-1/(q-2)})$ and $\norm{r_k}=\Theta(k^{-(q-1)/(q-2)})$.

(b) $q=2$: $d_k^2-d_{k+1}^2\ge2\mu d_{k+1}^2$, i.e.\
$d_{k+1}^2\le d_k^2/(1+2\mu)$.

(c) $q<2$: from \eqref{eq:growth} and Cauchy--Schwarz,
$\norm{w_{k+1}}\ge\mu d_{k+1}^{\,q-1}$; but
$\norm{w_{k+1}}=\norm{y_k-y_{k+1}}\le d_k+d_{k+1}\le 2d_k$, so
$d_{k+1}\le(2/\mu)^{1/(q-1)}d_k^{1/(q-1)}$.
\end{proof}

\begin{remark}[The error-bound parametrization]\label{rem:eb}
In terms of the H\"olderian error bound
$d(x)\le c\,\dist(0,Mx)^p$ ($p\in(0,1]$), one has $q=1+1/p$ and the
trichotomy reads: $p\in(0,1)$ (flat minima):
$d_k=O(k^{-p/(1-p)})$, $\norm{r_k}=O(k^{-1/(1-p)})$; $p=1$ (metric
subregularity): linear; the formally dual regime $p>1$ (sharp minima)
corresponds to $q<2$ and superlinear convergence of order $p$. These are
the correct exponents; they correct threshold errors that appear when the
Fej\'er term $\norm{w_{k+1}}^2$ is used alone (it is dominated by the
monotonicity term $2\inner{w_{k+1}}{y_{k+1}-z}$ precisely in the flat
regime).
\end{remark}

\section{Numerical experiments}\label{sec:numerics}

All experiments use the rotation-block resolvents of
Definition~\ref{def:barrier}-type operators, which are exactly the worst
case behind the tight PPM rate \cite{guyang2020tight,bravo2019sharp}. Code
is available from the author upon request.

\subsection{Experiment 1: dense ladders --- the barrier is real}\label{sec:exp1}

We take the geometric angle ladder $\theta_j=\frac{\pi}8(1.08)^{-j}$,
$j=0,\dots,234$ (from $\pi/8$ down to $10^{-7}$; floor $s\approx10^{-7}\ll
1/K$ for all $K$ used: the barrier regime of Theorem~\ref{thm:barrier}),
equal initial mass, $\norm{y_0}=1$. Figure~\ref{fig:ladder} and
Table~\ref{tab:ladder} report PPM, raw AA(3), and certified AA(3)
(Algorithm~1, $\delta=0$, $m=3$). All curves have log-log slope
$\approx-1/2$: on dense spectra AA does \emph{not} improve the rate, only
the constant ($\approx1.5\times$ at equal iteration count). Since the
problem is linear, the certified scheme accepts every probe
(Corollary~\ref{cor:linsafe}): its output coincides with raw AA and keeps
the PPM floor intact (verified pointwise); even after paying the factor-two
certification overhead it still edges PPM at equal oracle cost. This is
precisely the prediction of Theorems~\ref{thm:barrier}
and~\ref{thm:mainlb}: since the ladder's floor is below $1/K$, no
residual-polynomial method can beat the $1/(K+1)$ envelope here (up to the
$\log K$ slack of the roots-of-unity instance,
Remark~\ref{rem:artifact}), and AA's greedy per-step rule lands at the PPM
end of it.

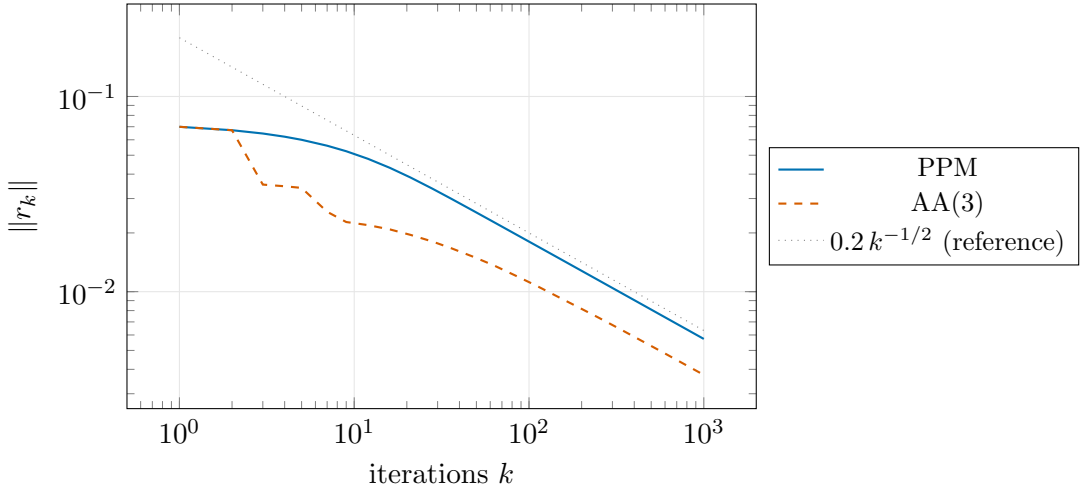
\begin{figure}[htbp]
\centering
\begin{tikzpicture}
\begin{loglogaxis}[width=0.60\textwidth,height=0.42\textwidth,
  xlabel={iterations $k$},ylabel={$\norm{r_k}$},
  legend style={at={(1.02,0.5)},anchor=west,font=\small},
  grid=major,grid style={gray!20}]
\addplot[thick,cblue] coordinates {
(1,6.9925e-02) (2,6.7176e-02) (3,6.4612e-02) (4,6.2220e-02) (5,5.9987e-02)
(7,5.5951e-02) (9,5.2421e-02) (12,4.7916e-02) (16,4.3086e-02) (21,3.8444e-02)
(28,3.3711e-02) (36,2.9888e-02) (48,2.5963e-02) (63,2.2702e-02) (83,1.9804e-02)
(110,1.7219e-02) (145,1.5009e-02) (191,1.3084e-02) (251,1.1418e-02)
(331,9.9464e-03) (437,8.6586e-03) (575,7.5497e-03) (759,6.5721e-03)
(1000,5.7263e-03)};
\addlegendentry{PPM}
\addplot[thick,cverm,dashed] coordinates {
(1,6.9925e-02) (2,6.7176e-02) (3,3.5368e-02) (4,3.4673e-02) (5,3.4013e-02)
(7,2.5667e-02) (9,2.2744e-02) (12,2.1925e-02) (16,2.0844e-02) (21,1.9545e-02)
(28,1.8071e-02) (36,1.6681e-02) (48,1.5094e-02) (63,1.3628e-02) (83,1.2131e-02)
(110,1.0740e-02) (145,9.4550e-03) (191,8.3360e-03) (251,7.3482e-03)
(331,6.4387e-03) (437,5.6247e-03) (575,4.9239e-03) (759,4.2933e-03)
(1000,3.7463e-03)};
\addlegendentry{AA(3)}
\addplot[thin,gray,dotted] coordinates {
(1,2.0000e-01) (2,1.4142e-01) (3,1.1547e-01) (4,1.0000e-01) (5,8.9443e-02)
(7,7.5593e-02) (9,6.6667e-02) (12,5.7735e-02) (16,5.0000e-02) (21,4.3644e-02)
(28,3.7796e-02) (36,3.3333e-02) (48,2.8868e-02) (63,2.5198e-02) (83,2.1953e-02)
(110,1.9069e-02) (145,1.6609e-02) (191,1.4471e-02) (251,1.2624e-02)
(331,1.0993e-02) (437,9.5673e-03) (575,8.3406e-03) (759,7.2595e-03)
(1000,6.3246e-03)};
\addlegendentry{$0.2\,k^{-1/2}$ (reference)}
\end{loglogaxis}
\end{tikzpicture}
\caption{Dense rotation ladder (floor $10^{-7}$): PPM and AA(3)
(Experiment~1). All slopes $\approx-1/2$: the adaptive barrier of
Theorem~\ref{thm:mainlb} bars any rate gain; only constants improve. On
this linear instance the output of certified AA(3) (Algorithm~1) coincides
with the AA(3) curve---acceptance is $100\%$ by
Corollary~\ref{cor:linsafe}---at two oracle evaluations per iteration; see
Table~\ref{tab:ladder} for the equal-cost comparison.}
\label{fig:ladder}
\end{figure}

\begin{table}[htbp]
\centering\small
\caption{Dense ladder, Experiment~1. Certified AA(3) costs $2$ evaluations
per iteration; the comparison at equal oracle cost reads it at iteration
$500$ ($1000$ evaluations) against PPM at iteration $1000$.}
\label{tab:ladder}
{\small
\begin{tabular}{lcccc}
\toprule
method & $\norm{r_{500}}$ & $\norm{r_{1000}}$ & log-log slope &
acceptance rate\\
\midrule
PPM & $8.10\cdot10^{-3}$ & $5.73\cdot10^{-3}$ & $-0.499$ & ---\\
AA(3) raw & $5.27\cdot10^{-3}$ & $3.75\cdot10^{-3}$ & $-0.486$ & ---\\
certified AA(3) (output) & $5.27\cdot10^{-3}$ (iters $500$/$1000$) &
--- & $-0.470$ & $1.00$\\
\bottomrule
\end{tabular}}
\end{table}

\subsection{Experiment 2: the extremal instance --- exactness to machine
precision}\label{sec:expExtremal}

On the extremal instance of Definition~\ref{def:extremal} we verify the
three claims of Theorem~\ref{thm:exactminimax} numerically. (i)~The
instance minimax, computed by solving the Gram system underlying
\eqref{eq:residsq}, matches $1/(K+1)$ to machine precision for
$K=1,\dots,100$, and the masses satisfy $w_0=w_{n-1}\to4/\pi^2$
(Table~\ref{tab:extremal}). (ii)~On the fixed instance with $K=30$,
Figure~\ref{fig:extremal} plots PPM, AA(3) and the averaged reflection
against the per-step floor $1/\sqrt{31(k+1)}$: the averaged reflection
\emph{coincides with the floor at every step} (max deviation
$6\cdot10^{-16}$), confirming that the degree-$k$ Fej\'er polynomial is
the instance-optimal degree-$k$ polynomial; AA(3) respects the floor and
sits between it and PPM; PPM is the slowest curve, as predicted by its
concentration on the near-fixed blocks. (iii)~Full-memory AA coincides
with the floor (it computes the instance-optimal polynomial by
construction): on the hardest instance, optimal adaptivity \emph{is}
averaged reflection.

\begin{table}[htbp]
\centering\small
\caption{The extremal instance (Experiment~2): instance minimax versus
$1/(K+1)$, and mass concentration at the two nearest nodes.}
\label{tab:extremal}
\begin{tabular}{rccccc}
\toprule
$K$ & instance minimax & $1/(K+1)$ & abs.\ error & $w_0$ & $w_0+w_{n-1}$\\
\midrule
$5$   & $0.1666666667$ & $0.1666666667$ & $2\cdot10^{-16}$ & $0.386700$ & $0.773400$\\
$10$  & $0.0909090909$ & $0.0909090909$ & $3\cdot10^{-16}$ & $0.397181$ & $0.794362$\\
$20$  & $0.0476190476$ & $0.0476190476$ & $2\cdot10^{-16}$ & $0.401539$ & $0.803079$\\
$30$  & $0.0322580645$ & $0.0322580645$ & $3\cdot10^{-17}$ & $0.403046$ & $0.806091$\\
$50$  & $0.0196078431$ & $0.0196078431$ & $7\cdot10^{-17}$ & $0.403887$ & $0.807774$\\
$100$ & $0.0099009901$ & $0.0099009901$ & $1\cdot10^{-16}$ & $0.404438$ & $0.808876$\\
$200$ & $0.0049751244$ & $0.0049751244$ & $4\cdot10^{-16}$ & $0.404713$ & $0.809426$\\
\bottomrule
\end{tabular}
\end{table}

\begin{figure}[htbp]
\centering
\begin{tikzpicture}
\begin{semilogyaxis}[width=0.55\textwidth,height=0.40\textwidth,
  xlabel={$k$},ylabel={$\norm{r_k}$},
  legend style={at={(1.02,0.5)},anchor=west,font=\small},
  grid=major,grid style={gray!20}]
\addplot[thick,cblue,mark=*,mark size=1.2pt] coordinates {
(1,1.2700e-01) (2,1.0370e-01) (3,8.9803e-02) (4,8.0322e-02) (5,7.3324e-02)
(6,6.7884e-02) (7,6.3500e-02) (8,5.9868e-02) (9,5.6796e-02) (10,5.4153e-02)
(11,5.1848e-02) (12,4.9814e-02) (13,4.8002e-02) (14,4.6374e-02) (15,4.4901e-02)
(16,4.3561e-02) (17,4.2333e-02) (18,4.1204e-02) (19,4.0161e-02) (20,3.9193e-02)
(21,3.8292e-02) (22,3.7450e-02) (23,3.6662e-02) (24,3.5921e-02) (25,3.5223e-02)
(26,3.4565e-02) (27,3.3942e-02) (28,3.3352e-02) (29,3.2791e-02) (30,3.2258e-02)};
\addlegendentry{averaged reflection $=$ floor (Thm.~\ref{thm:exactminimax})}
\addplot[thick,cverm,dashed,mark=square*,mark size=1.2pt] coordinates {
(1,1.2700e-01) (2,1.0999e-01) (3,9.4660e-02) (4,8.4003e-02) (5,7.7771e-02)
(6,7.3879e-02) (7,7.1496e-02) (8,6.9291e-02) (9,6.7280e-02) (10,6.5675e-02)
(11,6.4326e-02) (12,6.2964e-02) (13,6.1578e-02) (14,6.0332e-02) (15,5.9482e-02)
(16,5.8615e-02) (17,5.7695e-02) (18,5.6692e-02) (19,5.5937e-02) (20,5.5363e-02)
(21,5.4740e-02) (22,5.3980e-02) (23,5.3287e-02) (24,5.2757e-02) (25,5.2332e-02)
(26,5.1782e-02) (27,5.1209e-02) (28,5.0711e-02) (29,5.0341e-02) (30,4.9965e-02)};
\addlegendentry{AA(3)}
\addplot[thick,cblgr,dotted,mark=triangle*,mark size=1.4pt] coordinates {
(1,1.2700e-01) (2,1.0999e-01) (3,1.0040e-01) (4,9.3918e-02) (5,8.9098e-02)
(6,8.5305e-02) (7,8.2202e-02) (8,7.9592e-02) (9,7.7349e-02) (10,7.5391e-02)
(11,7.3657e-02) (12,7.2107e-02) (13,7.0706e-02) (14,6.9432e-02) (15,6.8265e-02)
(16,6.7190e-02) (17,6.6195e-02) (18,6.5269e-02) (19,6.4404e-02) (20,6.3594e-02)
(21,6.2833e-02) (22,6.2114e-02) (23,6.1436e-02) (24,6.0792e-02) (25,6.0181e-02)
(26,5.9600e-02) (27,5.9045e-02) (28,5.8516e-02) (29,5.8009e-02) (30,5.7524e-02)};
\addlegendentry{PPM}
\addplot[thick,crp,dashdotted] coordinates {
(1,1.2700e-01) (2,1.0370e-01) (3,8.9803e-02) (4,8.0322e-02) (5,7.3324e-02)
(6,6.7884e-02) (7,6.3500e-02) (8,5.9868e-02) (9,5.6796e-02) (10,5.4153e-02)
(11,5.1848e-02) (12,4.9814e-02) (13,4.8002e-02) (14,4.6374e-02) (15,4.4901e-02)
(16,4.3561e-02) (17,4.2333e-02) (18,4.1204e-02) (19,4.0161e-02) (20,3.9193e-02)
(21,3.8292e-02) (22,3.7450e-02) (23,3.6662e-02) (24,3.5921e-02) (25,3.5223e-02)
(26,3.4565e-02) (27,3.3942e-02) (28,3.3352e-02) (29,3.2791e-02) (30,3.2258e-02)};
\addlegendentry{per-step floor $1/\sqrt{31(k+1)}$}
\end{semilogyaxis}
\end{tikzpicture}
\caption{The extremal instance of Definition~\ref{def:extremal}, $K=30$
(Experiment~2): the averaged reflection coincides at every step with the
per-step instance floor $1/\sqrt{31(k+1)}$
(Theorem~\ref{thm:exactminimax}(iii)), AA(3) respects the floor, and PPM
is slowest. Full-memory AA coincides with the floor curve.}
\label{fig:extremal}
\end{figure}
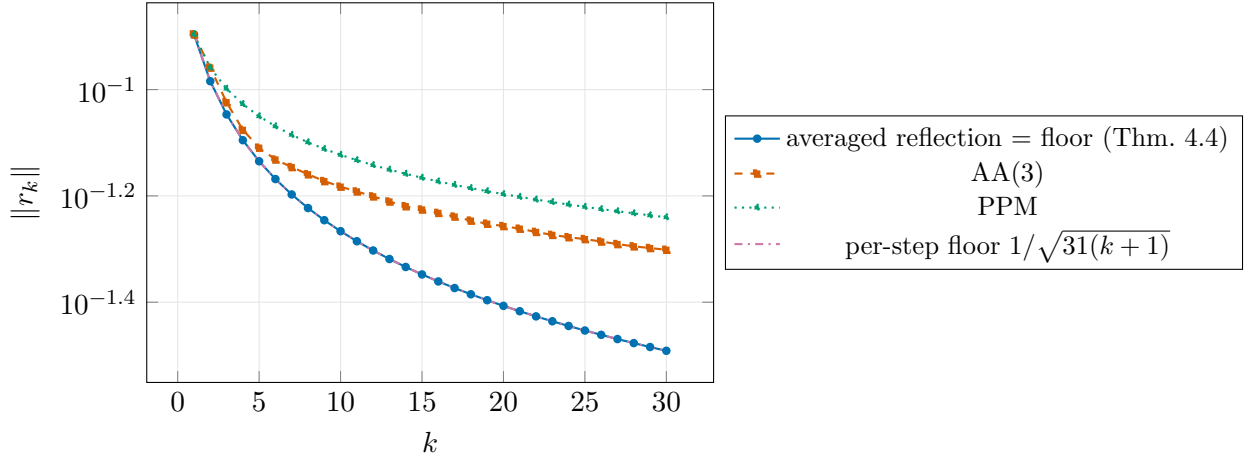

\subsection{Experiment 3: the roots-of-unity barrier operator}
\label{sec:exp2}

On the operator $M_K$ of Definition~\ref{def:barrier} with $K=30$ and the
optimal-mass initial point $y_0^\star$ of \eqref{eq:masses},
Figure~\ref{fig:t3} plots AA(3) and PPM against the instance floor
$1/\Lambda_{30}$: by Theorem~\ref{thm:mainlb}(i) applied to this operator,
\emph{every} degree-$\le30$ polynomial method satisfies
$\norm{r_k}\ge1/\Lambda_{30}$ at every step $k\le30$, and both curves
respect it. The floor bites: AA(3) ends within a factor $4.6$ of
$1/\Lambda_{30}=0.0137$. Note that the extremal instance of
Definition~\ref{def:extremal} is strictly harder at the same budget: its
instance minimax is $1/31=0.0323$, a factor $2.4$ above $1/\Lambda_{30}$,
and the factor grows like $\Theta(\log K)$---the $\csc$ versus $\csc^2$
artifact of Remark~\ref{rem:artifact}. (The per-budget statement of
Theorem~\ref{thm:mainlb}---for each budget $k$ its own hard instance
$M_k$---gives the growing bound $1/\Lambda_k\asymp1/(k\log k)$.)

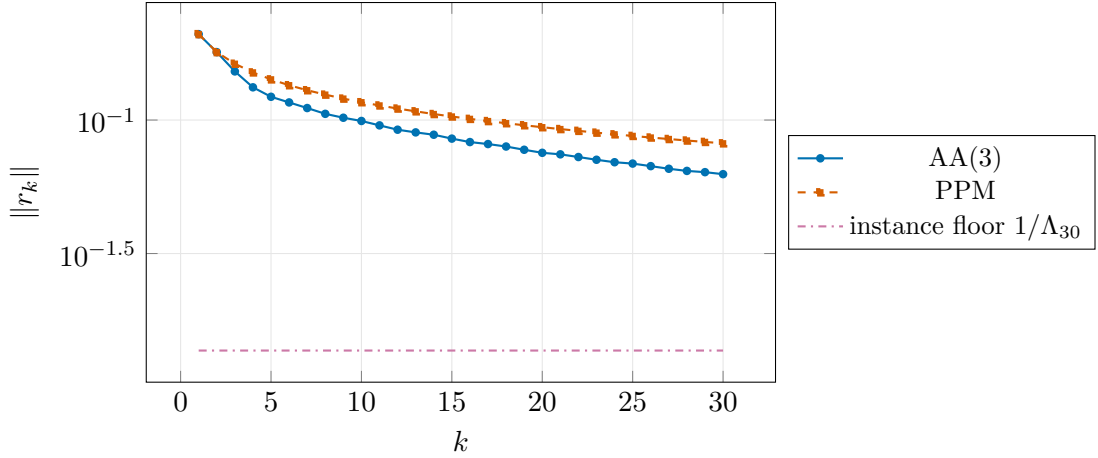
\begin{figure}[htbp]
\centering
\begin{tikzpicture}
\begin{semilogyaxis}[width=0.60\textwidth,height=0.40\textwidth,
  xlabel={$k$},ylabel={$\norm{r_k}$},
  legend style={at={(1.02,0.5)},anchor=west,font=\small},
  grid=major,grid style={gray!20}]
\addplot[thick,cblue,mark=*,mark size=1.2pt] coordinates {
(1,0.2097) (2,0.1796) (3,0.1521) (4,0.1326) (5,0.1222) (6,0.1164) (7,0.1109)
(8,0.1055) (9,0.1020) (10,0.0992) (11,0.0955) (12,0.0920) (13,0.0899)
(14,0.0881) (15,0.0852) (16,0.0827) (17,0.0813) (18,0.0796) (19,0.0774)
(20,0.0754) (21,0.0744) (22,0.0727) (23,0.0710) (24,0.0695) (25,0.0687)
(26,0.0672) (27,0.0657) (28,0.0645) (29,0.0638) (30,0.0627)};
\addlegendentry{AA(3)}
\addplot[thick,cverm,dashed,mark=square*,mark size=1.2pt] coordinates {
(1,0.2097) (2,0.1796) (3,0.1625) (4,0.1508) (5,0.1420) (6,0.1350) (7,0.1293)
(8,0.1245) (9,0.1203) (10,0.1166) (11,0.1134) (12,0.1104) (13,0.1078)
(14,0.1053) (15,0.1031) (16,0.1010) (17,0.0991) (18,0.0973) (19,0.0956)
(20,0.0940) (21,0.0925) (22,0.0911) (23,0.0897) (24,0.0884) (25,0.0872)
(26,0.0860) (27,0.0849) (28,0.0838) (29,0.0828) (30,0.0818)};
\addlegendentry{PPM}
\addplot[thick,crp,dashdotted] coordinates {
(1,0.0137) (2,0.0137) (3,0.0137) (4,0.0137) (5,0.0137) (6,0.0137) (7,0.0137)
(8,0.0137) (9,0.0137) (10,0.0137) (11,0.0137) (12,0.0137) (13,0.0137)
(14,0.0137) (15,0.0137) (16,0.0137) (17,0.0137) (18,0.0137) (19,0.0137)
(20,0.0137) (21,0.0137) (22,0.0137) (23,0.0137) (24,0.0137) (25,0.0137)
(26,0.0137) (27,0.0137) (28,0.0137) (29,0.0137) (30,0.0137)};
\addlegendentry{instance floor $1/\Lambda_{30}$}
\end{semilogyaxis}
\end{tikzpicture}
\caption{The barrier operator of Definition~\ref{def:barrier}, $K=30$,
optimal-mass $y_0^\star$ (Experiment~3): AA(3) and PPM against the instance
floor $1/\Lambda_{30}$, which by Theorem~\ref{thm:mainlb} bounds every
degree-$\le30$ polynomial method from below at every step on this instance.}
\label{fig:t3}
\end{figure}

\subsection{Experiment 4: the phase transition crossover}\label{sec:exp3}

For spectra with two-sided floor $\theta\in[\delta,\pi-\delta]$ (i.e.\
$\dist(1,\sigma)\ge\sin\delta$), $K=50$, we compare the Jackson estimator of
Theorem~\ref{thm:jackson} (best over
$N\in\{\lfloor(K+2)/8\rfloor,\lfloor(K+2)/4\rfloor,\lfloor(K+2)/2\rfloor\}$)
against the universal Fej\'er envelope $1/(K+1)=0.0196$.

\begin{table}[htbp]
\centering\small
\caption{Jackson vs.\ Fej\'er on floored spectra, $K=50$
(Experiment~4). The crossover sits at $\delta K\approx10$.}
\label{tab:jackson}
\begin{tabular}{cccc}
\toprule
$\delta K$ & Jackson $\max_\theta\abs{(1-\zeta)q_J(\zeta)}$ &
Fej\'er $1/(K+1)$ & ratio\\
\midrule
$\pi$ & $5.27\cdot10^{-2}$ & $1.96\cdot10^{-2}$ & $2.69$\\
$10$ & $8.23\cdot10^{-3}$ & $1.96\cdot10^{-2}$ & $0.42$\\
$50$ & $1.90\cdot10^{-3}$ & $1.96\cdot10^{-2}$ & $0.097$\\
\bottomrule
\end{tabular}
\end{table}

\subsection{Experiment 5: the nonlinear family $M=S+N_C$}\label{sec:expNonlin}

We illustrate Section~\ref{sec:beyond} on $M=S+N_C$ with $S$ the extremal
skew-adjoint of Definition~\ref{def:extremal} at $K=29$ (chosen odd so
that no block is the frozen block of Definition~\ref{def:extremal}) and
$C$ the Euclidean ball of radius $0.6$. The resolvent is computed through
the exact variational characterization $x=P_C(x+\tau(y-x-Sx))$, solved to
$\norm{\cdot}_\infty\le10^{-13}$; the solution of $y\in x+Sx+N_C(x)$ is
verified to satisfy the normal-cone inclusion with orthogonal defect
$10^{-12}$ (the constraint is active: $\norm{J(y_0)}=0.6$, multiplier
$\lambda=0.65$). Table~\ref{tab:nonlin} reports PPM, the averaged
reflection, and AA(3): the universal bound $1/(k+1)$ holds everywhere, and
on this particular instance PPM is \emph{faster} than the averaged
reflection---the ball constraint removes the near-fixed blocks that make
the extremal instance hard (the reflection orbit leaves the boundary after
the first step), illustrating Remark~\ref{rem:genuine}: the worst case
over $\mathfrak F$ is linear, but individual nonlinear instances can be
easier. For $C=V$ a subspace (first $2m$ coordinates, $m=15$), the
compressed resolvent of Proposition~\ref{prop:subspace} is verified to be
skew-adjoint-shaped: $S_V$ is exactly skew-symmetric and the eigenvalues
of the compressed resolvent satisfy $\abs{2\zeta-1}=1$ to
$2\cdot10^{-15}$---the family $\{S+N_V\}$ stays inside the skew-adjoint
class---and the averaged reflection respects $1/(k+1)$ throughout.

\begin{table}[htbp]
\centering\small
\caption{The nonlinear instance $M=S+N_C$ of Experiment~5 ($K=29$, $C$ the
ball of radius $0.6$): residuals $\norm{r_k}$ against the universal bound
$1/(k+1)$. All values verified against the exact variational solution.}
\label{tab:nonlin}
\begin{tabular}{rcccc}
\toprule
$k$ & PPM & averaged reflection & AA(3) & bound $1/(k+1)$\\
\midrule
$1$  & $8.73\cdot10^{-2}$ & $2.94\cdot10^{-1}$ & $2.94\cdot10^{-1}$ & $5.00\cdot10^{-1}$\\
$2$  & $6.27\cdot10^{-2}$ & $1.79\cdot10^{-1}$ & $1.74\cdot10^{-1}$ & $3.33\cdot10^{-1}$\\
$5$  & $4.48\cdot10^{-2}$ & $9.82\cdot10^{-2}$ & $8.49\cdot10^{-2}$ & $1.67\cdot10^{-1}$\\
$10$ & $3.50\cdot10^{-2}$ & $5.85\cdot10^{-2}$ & $5.21\cdot10^{-2}$ & $9.09\cdot10^{-2}$\\
$15$ & $2.95\cdot10^{-2}$ & $4.19\cdot10^{-2}$ & $3.83\cdot10^{-2}$ & $6.25\cdot10^{-2}$\\
$20$ & $2.58\cdot10^{-2}$ & $3.30\cdot10^{-2}$ & $3.02\cdot10^{-2}$ & $4.76\cdot10^{-2}$\\
$25$ & $2.31\cdot10^{-2}$ & $2.75\cdot10^{-2}$ & $2.49\cdot10^{-2}$ & $3.85\cdot10^{-2}$\\
$29$ & $2.13\cdot10^{-2}$ & $2.42\cdot10^{-2}$ & $2.16\cdot10^{-2}$ & $3.33\cdot10^{-2}$\\
\bottomrule
\end{tabular}
\end{table}

\subsection{Experiment 6: safeguarding --- acceptance profile and linear
safety}\label{sec:exp4}

(i) On $4000$ random linear monotone systems (random symmetric psd plus
skew parts, dimensions $2$--$8$) and $400$ random projected-nonlinear
systems $J=P_C\circ(I+A)^{-1}$ with $A$ monotone, we tracked the descent
ratio $\norm{r(\hat y)}/\norm{r_k}$ of raw AA($2$): the maximum observed
was $1.000$ (never above $1$), as Corollary~\ref{cor:linsafe} mandates in
the linear case and as is empirically typical even nonlinearly on monotone
problems (Remark~\ref{rem:zob}). (ii) On the three-block spectrum
$\{0.9,0.3,0.05\}$ (floored, escape regime), certified AA(3) reaches
$\norm{r_{120}}=1.90\cdot10^{-2}$ against PPM's $2.14\cdot10^{-2}$ at
equal oracle cost ($240$ evaluations). On both this instance and the dense
ladder of Experiment~1 the acceptance rate is $1.00$---the safeguard never
fires on linear problems, exactly as Corollary~\ref{cor:linsafe}
predicts---the certified output equals raw AA and dominates PPM at equal
oracle cost, and the floor $\norm{r(b_k)}\le\norm{r^c_k}$ of
Theorem~\ref{thm:cert} holds pointwise.

\subsection{Experiment 7: H\"older growth rates}\label{sec:exp5}

PPM on $f(x)=\abs{x}^\gamma/\gamma$ in $\R$ (resolvent by bisection to
machine precision), $40\,000$ iterations, slopes on
$k\in[2\cdot10^4,4\cdot10^4]$:

\begin{table}[htbp]
\centering\small
\caption{H\"older growth, Experiment~7 vs.\ Theorem~\ref{thm:growth}(a):
predicted exponents $-1/(\gamma-2)$ for $d_k$ and $-(\gamma-1)/(\gamma-2)$
for $\norm{r_k}$.}
\label{tab:growth}
\begin{tabular}{cccccc}
\toprule
$\gamma$ & $q$ & slope $d_k$ & predicted & slope $\norm{r_k}$ & predicted\\
\midrule
$4$ & $4$ & $-0.500$ & $-0.500$ & $-1.500$ & $-1.500$\\
$6$ & $6$ & $-0.250$ & $-0.250$ & $-1.250$ & $-1.250$\\
\bottomrule
\end{tabular}
\end{table}

\subsection{Experiment 8: averaged reflection and the Fej\'er
constant}\label{sec:exp6}

(i) The Fej\'er residual polynomial attains
$\max_{\abs{u}=1}\abs{(1-u)q_K^\star(u)}=2/(K+1)$ exactly:
$0.181818$ vs.\ $2/11$ at $K=10$, $0.019802$ vs.\ $2/101$ at $K=100$
(Table~\ref{tab:fejer}). (ii) On a random \emph{nonlinear} monotone
example ($J=P_{[-0.3,0.3]^4}\circ(I+A)^{-1}$, $A$ monotone, $K=60$), the
averaged-reflection estimator \eqref{eq:avredef} gives
$\norm{r(\hat y_K)}=0.0603\le d_0/(K+1)=0.1283$: the telescoping bound of
Theorem~\ref{thm:avref} holds beyond the linear theory.

\begin{table}[htbp]
\centering\small
\caption{Fej\'er extremal values (Experiment~8).}
\label{tab:fejer}
\begin{tabular}{ccc}
\toprule
$K$ & $\max_{\abs{u}=1}\abs{(1-u)q_K^\star(u)}$ & $2/(K+1)$\\
\midrule
$10$ & $0.181818$ & $0.181818$\\
$100$ & $0.019802$ & $0.019802$\\
\bottomrule
\end{tabular}
\end{table}

\section{Discussion and open problems}\label{sec:discussion}

\textbf{What the map says.} The complexity of residual-polynomial
acceleration of PPM is now pinned exactly on three levels. Uniformly over
all maximal monotone operators, the answer is exactly $d_0/(K+1)$,
achieved by the elementary averaged-reflection estimator
\eqref{eq:avredef} and matched by the extremal instance of
Definition~\ref{def:extremal} against every polynomial method, adaptive or
not (Theorem~\ref{thm:exactminimax}, Corollary~\ref{cor:exactV}); on the
hardest instance the optimal adaptive scheme is the non-adaptive Fej\'er
kernel itself. The obstruction is spectral, not algorithmic: the
adversarial mass concentrates at the critical scale
$\dist(\zeta,1)\asymp1/K$ with the critical mass exponent $2$
(Remark~\ref{rem:artifact}). Structure helps exactly when the spectrum
clears the critical scale: $sK\to\infty$ opens the Jackson/Chebyshev
escape route (Theorem~\ref{thm:jackson}), and $sK=O(1)$ closes it with the
exact barrier $1/(K+1)$ (Theorem~\ref{thm:barrier}); both sides extend to
normal operators (Theorem~\ref{thm:normal}), and the exact value extends
to the nonlinear family $S+N_C$ (Theorem~\ref{thm:nonlinminimax}).
Safeguarding costs a factor two, and only on nonlinear problems
(Corollaries~\ref{cor:linsafe} and~\ref{cor:optimal2}).

\textbf{Open problems.}
(1) \emph{Greedy vs.\ minimax on dense spectra}: empirically AA tracks
$\Theta(1/\sqrt{k})$ with a small constant gain on dense ladders, far from
the instance minimax; is there a proved per-step suboptimality statement
for greedy residual minimization, or a better adaptive rule closing the
gap toward the instance minimax?
(2) \emph{Memory tradeoffs}: on the extremal instance, limited memory
provably loses (AA(3) sits above the per-step floor in
Figure~\ref{fig:extremal}); a sharp memory-versus-rate lower bound,
complementing the Chebyshev-number upper bounds $\tau_m$ of the arc, is
open.
(3) \emph{Algorithmic escape}: the Jackson polynomial of
Theorem~\ref{thm:jackson} is tuned to the floor $s$; an adaptive
diagnostic that detects the floor online and switches between the Fej\'er
and Jackson regimes without knowing $s$ would make the escape side
algorithmic.
(4) \emph{Non-normal and genuinely nonlinear sharpness}: the exact value
$d_0/(K+1)$ is proved for skew-adjoint (hence normal) instances and, by
inclusion, for the family $S+N_C$; for non-normal linear operators the
escape side degrades by a Kreiss factor (Remark~\ref{rem:nonnormal}), and
for genuinely nonlinear $J$ the polynomial reduction fails
(Remark~\ref{rem:genuine}); sharp constants in both directions are open.
(5) \emph{Nonlinear polynomial theory}: the reduction of
Lemma~\ref{lem:polyred} is linear; nothing comparable is known for
genuinely nonlinear $J$, where the residual of a combination is not the
combination of residuals (Remark~\ref{rem:nonlin}).

\textbf{On the stochastic setting.} A stochastic extension is \emph{not} a
routine modification, and we prefer to say so plainly rather than to claim
unsupported rates. With noisy resolvents
$\widetilde J(y)=J(y)+\xi$, the AA candidate
$\hat y=\sum_i\gamma_i\widetilde x_i$ inherits the noise
$\sum_i\gamma_i\xi_i$ with amplification $\sum_i\abs{\gamma_i}$---the
Lebesgue constant of the window---which is unbounded in general and large
exactly when AA extrapolates aggressively (the regime of
Proposition~\ref{prop:nohist}). Consequently, no variance-reduction-free
stochastic rate can hold without controlling $\sum_i\abs{\gamma_i}$ (e.g.\
by coefficient-capped or ridge-regularized AA), and the acceptance
certificate of Theorem~\ref{thm:cert} must itself be replaced by a
high-probability variant. We regard a variance-aware certified stochastic
AA-PPM as an open problem, not a corollary.

\appendix
\section{A sequence lemma}\label{app:seq}

\begin{lemma}\label{lem:seq}
Let $u_k\ge0$ satisfy $u_k-u_{k+1}\ge c\,u_{k+1}^{\alpha}$ with $c>0$,
$\alpha>1$. Then
\[
u_k\ \le\ \big(c(\alpha-1)\,k\big)^{-1/(\alpha-1)}
\qquad\text{for all }k\ge1 .
\]
\end{lemma}
\begin{proof}
The function $\phi(u)=u^{1-\alpha}$ is convex and decreasing with
$\phi'(u)=(1-\alpha)u^{-\alpha}$. Hence
\[
\phi(u_{k+1})-\phi(u_k)
=(u_{k+1}^{1-\alpha}-u_k^{1-\alpha})
\ge (1-\alpha)u_{k+1}^{-\alpha}(u_{k+1}-u_k)
\ge (\alpha-1)u_{k+1}^{-\alpha}\,c\,u_{k+1}^{\alpha}
=c(\alpha-1).
\]
Telescoping,
$u_k^{1-\alpha}\ge u_0^{1-\alpha}+c(\alpha-1)k\ge c(\alpha-1)k$, and
inverting gives the claim.
\end{proof}

\end{document}